\documentclass[a4paper]{amsart}
\usepackage{amsmath,amsfonts,amssymb,amsthm}
\usepackage{enumerate}
\newtheorem{theorem}{Theorem}[section]
\newtheorem{lemma}[theorem]{Lemma}
\newtheorem{proposition}[theorem]{Proposition}
\newtheorem{corollary}[theorem]{Corollary}
\theoremstyle{remark}
\newtheorem{remark}[theorem]{Remark}
\theoremstyle{definition}

\newcommand{\Mod}[1]{\, (#1)}

\newcommand{\eps}{\varepsilon}
\title[Siegel zeros and primes]{Siegel zeros, twin primes, Goldbach's conjecture, and primes in short intervals}
\author{Kaisa Matom\"aki}
\address{Department of Mathematics and Statistics, University of Turku, 20014 Turku, Finland}
\email{ksmato@utu.fi}
\author{Jori Merikoski}
\address{Mathematical Institute, University of Oxford, Andrew Wiles Building, Radcliffe Observatory Quarter (550), Woodstock Road, Oxford, OX2 6GG}
\email{jori.merikoski@maths.ox.ac.uk}
\subjclass{11M20, 11N36, 11P32}
\keywords{Siegel zeros, twin primes, Goldbach's conjecture}

\begin{document}

\begin{abstract} 
We study the distribution of prime numbers under the unlikely assumption that Siegel zeros exist. In particular we prove for
\[
\sum_{n \leq X} \Lambda(n) \Lambda(\pm n+h)
\]
an asymptotic formula which holds uniformly for $h = O(X)$. Such an asymptotic formula has been previously obtained only for fixed $h$ in which case our result quantitatively improves those of Heath-Brown (1983) and Tao and Ter\"av\"ainen (2021). 

Since our main theorems work also for large $h$ we can derive new results concerning connections between Siegel zeros and the Goldbach conjecture and between Siegel zeros and primes in almost all very short intervals.
\end{abstract}

\maketitle

\section{Introduction}
While the proof of the twin prime conjecture is a distant goal, Heath-Brown~\cite{hbsiegel} proved in 1983 that if there are infinitely many Siegel zeros, then there are infinitely many twin primes. More precisely Heath-Brown showed that if, for a Dirichlet character $\chi \pmod{q}$, the Dirichlet $L$-function $L(s,\chi)$ has a real zero at $s = \beta_0$ with
\begin{equation} \label{eq:szeroassumption}
\beta_0 = 1- \frac{1}{\eta \log q}
\end{equation}
for some $\eta \geq 10$, then, for any $h \geq 1$ and $X \in [q^{250}, q^{500}]$, one has
\[
\sum_{n \leq X} \Lambda(n) \Lambda(n+h) = X\mathfrak{S}_h + O_h\left(\frac{X}{\log \log \eta}\right),
\]
where 
\[
\mathfrak{S}_h := \mathbf{1}_{2 \mid h}  \cdot 2\prod_{p > 2} \left(1-\frac{1}{(p-1)^2}\right) \prod_{\substack{p \mid h \\ p > 2}} \left(1+\frac{1}{p-2}\right) \asymp \mathbf{1}_{2 \mid h} \frac{h}{\varphi(h)}.
\]
Actually Heath-Brown proved a more general result for sums of the form
\[
\sum_{n \leq X} \Lambda(\alpha_1 n + \beta_1) \Lambda(\alpha_2 n + \beta_2).
\]

Very recently Heath-Brown's result was quantitatively improved by Tao and Ter\"av\"ainen~\cite{TT} who showed that, for any $h \geq 1$ and $X \in [q^{41/2+\varepsilon}, q^{\eta^{1/2}}]$, one has
\begin{equation}
\label{eq:H-LTT}
\sum_{n \leq X} \Lambda(n) \Lambda(n+h) = X\mathfrak{S}_h + O_h\left(\frac{X}{\log^{1/20} \eta}\right).
\end{equation}
This is a special case of their more general theorem on "Hardy-Littlewood-Chowla" type correlations
\begin{equation}
\label{eq:TTgeneral}
\sum_{n \leq X} \Lambda(n) \Lambda(n+h) \lambda(n+h_1) \dotsm \lambda(n+h_k).
\end{equation}

In this paper we will prove~\eqref{eq:H-LTT} with better error term and with uniformity in the shift $h$. Furthermore, despite leading to stronger results, our proof is less involved. Before turning to the strongest formulations of our theorems, let us state two corollaries. 
The first corollary improves on~\eqref{eq:H-LTT}. 
\begin{corollary}
\label{cor:simpleMT}
Let $h \in \mathbb{N}$. Let $\chi$ be a primitive quadratic character modulo $q \geq 2$ and assume that $L(s,\chi)$ has a real zero $\beta_0$ such that
\[
\beta_0 = 1-\frac{1}{\eta \log q}
\]
for some $\eta \geq 10$.
\begin{enumerate}[(i)]
\item Let $C \geq 1$. For $X \in [q^{10}, q^{10 \log \eta}]$, we have 
\[
\sum_{n \leq X} \Lambda(n) \Lambda(n+h) = X\mathfrak{S}_h + O_{h, C}\left(X\exp(-C\sqrt{\log \eta})\right).
\]
\item Let $\varepsilon > 0$. For $X = q^V$ with $V \in [10 \log \eta, \eta^{1-\varepsilon}]$, we have
\[
\sum_{n \leq X} \Lambda(n) \Lambda(n+h) = X\mathfrak{S}_h + O_{h}\left(X\frac{\log^6 \eta}{\eta/V} \right).
\]
\end{enumerate}
\end{corollary}
Note that for $X \in [q^{10\log \eta}, q^{10\log^4 \eta}]$ the error term is $O(\eta^{-1}X\log^{10} \eta)$, where the dependency on $\eta$ is best that can be hoped for apart from the power of $\log \eta$.

Our main results are uniform with respect to $h$ which allows us to attack also the Goldbach conjecture. There has been recent activity on the relation between Siegel zeros and the Goldbach conjecture, see e.g.~\cite{F-ISiegGold,G-SSiegGold}. The asymptotic form of Goldbach's conjecture claims that, for all $h \geq 4$,
\begin{equation}
\label{eq:GoldAsymp}
\sum_{\substack{n_1, n_2 \leq h \\ n_1 + n_2 = h}} \Lambda(n_1) \Lambda(n_2) = (1+o(1)) \mathfrak{S}_h \cdot h.
\end{equation}
We show that if a weak form of~\eqref{eq:GoldAsymp} holds, then there cannot be Siegel zeros.
\begin{corollary}
\label{cor:Goldbach}
Let $\delta > 0$. There exists $\eta = \eta(\delta) \geq 100$ such that the following holds. Let $q \geq 2$ be such that there exists a primitive quadratic character $\chi \pmod{q}.$ Assume that there exists an even $h \in [q^{10}, q^{\eta^{99/100}}]$ such that $q \mid h$ and
\begin{equation}
\label{eq:WeakGoldbach}
\delta \mathfrak{S}_h \cdot h \leq \sum_{\substack{n_1, n_2 \leq h \\ n_1 + n_2 = h}} \Lambda(n_1) \Lambda(n_2) \leq (2-\delta) \mathfrak{S}_h \cdot h.
\end{equation}
Then the Dirichlet $L$-function $L(s, \chi)$ does not have an exceptional zero $\beta_0$ with
\[
\beta_0 \geq 1-\frac{1}{\eta \log q}.
\]
\end{corollary}
This improves on a recent result of Friedlander, Goldston, Iwaniec and Suriajaya~\cite{F-ISiegGold} who got a similar conclusion assuming that~\eqref{eq:WeakGoldbach} holds for several $h \equiv 0 \pmod{q}$. In fact, our result is even stronger and we only need the lower bound in~\eqref{eq:WeakGoldbach} if $\chi(-1)=-1$ and similarly only the upper bound in~\eqref{eq:WeakGoldbach} if $\chi(-1)=1$.

We now turn to the precise statements of our main results. When we work uniformly with respect to $h$, we get an additional main term when $(h, q)$ has size close to $q$. To explain why, let us consider the most transparent case $h = q$. It is well known that if there is a Siegel zero, then the primes fail to be equidistributed $\pmod{q}$, more precisely we have when $q, \chi,$ and $\beta_0$ are as in Corollary~\ref{cor:simpleMT} with $\eta$ large (see e.g.~\cite[Theorem 5.27]{IwKo}),
\[
\sum_{\substack{n \leq x \\ n \equiv a \pmod{q}}} \Lambda(n) = \frac{x}{\varphi(q)}\left(1 - \frac{\chi(a)}{\beta_0 x^{1-\beta_0}}\right) + O\left(x \exp\left(\frac{-c \log x}{\sqrt{\log x} + \log q}\right) \log^4 q \right). 
\]
Hence, when $\eta$ is large, the residue classes $a \pmod{q}$ with $\chi(a) = -1$ contain about twice as many primes as one would expect, whereas the residue classes $a \pmod{q}$ with $\chi(a) = 1$ contain very few primes. Consequently one would expect that
\[
\sum_{n \leq X} \Lambda(n) \Lambda(n+q) \approx 2 \mathfrak{S}_q X.
\]
The following general theorem confirms this intuition.
\begin{theorem}
\label{th:MTzero}
Let $C \geq 1$ and $\varepsilon > 0$. Let $\chi$ be a primitive quadratic character modulo $q \geq 2$. Write $q = 2^r q'$ with $r \geq 0$ and $2 \nmid q'$. Assume that $L(s,\chi)$ has a real zero $\beta_0$ such that
\[
\beta_0 = 1-\frac{1}{\eta \log q}
\]
for some $\eta \geq 10$. Let $0 \neq h = O(X)$. We have, for any $X = q^V$ with $V \geq 10$,
\[
\begin{split}
&\sum_{n \leq X} \Lambda(n) \Lambda(n+h) = X\mathfrak{S}_h\left(1 + \mathbf{1}_{\varphi(2^r) \mid h} (-1)^{\frac{h}{\varphi(2^r)}} \prod_{\substack{p \mid q' \\ p \nmid h}} \frac{-1}{p-2}\right) \\
&\qquad + O_{C, \varepsilon}\left(\frac{h}{\varphi(h)} X \left(\exp(-C \sqrt{V \log \eta}) + \exp(-C(\log X)^{3/5-\varepsilon}) + \frac{V (\log \eta)^6}{\eta} \right)\right).
\end{split}
\]
\end{theorem}
Note that the main term vanishes for some even $h$, for instance when $3 \mid q$ and $h = 2q/3$ (note that $q'$ is necessarily square-free (see e.g.~\cite[Section 3.3]{IwKo}) and so in this case $3 \nmid h$).

We get a similar result concerning Goldbach's conjecture.
\begin{theorem}
\label{th:Goldbach}
Let $C \geq 1$ and $\varepsilon > 0$. Let $\chi$ be a primitive quadratic character modulo $q \geq 2$. Write $q = 2^r q'$ with $r \geq 0$ and $2 \nmid q'$. Assume that $L(s,\chi)$ has a real zero $\beta_0$ such that
\[
\beta_0 = 1-\frac{1}{\eta \log q}.
\]
for some $\eta \geq 10$. Let $h \geq q^{10}$ be an integer. Writing $V := \frac{\log h}{\log q} \geq 10$, we have
\[
\begin{split}
&\sum_{\substack{n_1, n_2 \leq h \\ n_1 + n_2 = h}} \Lambda(n_1) \Lambda(n_2) = h\mathfrak{S}_h\left(1 + \chi(-1)\mathbf{1}_{\varphi(2^r) \mid h} (-1)^{\frac{h}{\varphi(2^r)}} \prod_{\substack{p \mid q' \\ p \nmid h}} \frac{-1}{p-2}\right) \\
&\qquad + O_{C, \varepsilon}\left(\frac{h}{\varphi(h)} h \left(\exp(-C \sqrt{V \log \eta}) + \exp(-C(\log h)^{3/5-\varepsilon}) + \frac{V (\log \eta)^6}{\eta} \right)\right).
\end{split}
\]
\end{theorem}

Corollaries~\ref{cor:simpleMT}(i) and~\ref{cor:Goldbach} immediately follow from Theorems~\ref{th:MTzero} and~\ref{th:Goldbach} since by Siegel's theorem (see e.g.~\cite[Theorem 11.14 combined with (11.10)]{MVBook})
\begin{equation}
\label{eq:Siegeleta}
\eta \ll_\varepsilon q^\varepsilon.
\end{equation}
For $X \geq q^{10 \log \eta}$ the quantity 
\[
\frac{1}{\eta} = \exp(-\log \eta) \gg \exp(-\sqrt{(\log q) (\log \eta)}) \gg \exp(-\sqrt{\log X}),
\]
dominates $\exp(-C(\log X)^{3/5-\varepsilon})$ and $\exp(-C V \sqrt{\log \eta})$, so also Corollary~\ref{cor:simpleMT}(ii) follows from Theorem~\ref{th:MTzero}.

We will also prove a corollary concerning the distribution of primes in almost all short intervals. A probabilistic model predicts that 
\begin{equation}
\label{eq:pntalmostallshort}
\sum_{y < p \leq y+ H} 1 = (1+o(1)) \frac{H}{\log X}.
\end{equation}
holds for almost all $y \in [X,2X]$ as soon as $H/\log X \to \infty$. Selberg~\cite{Selberg} has shown this for $H/\log^2 X \to \infty$ assuming the Riemann hypothesis. Assuming also the Pair correlation conjecture~\cite{montgomery}, Heath-Brown~\cite{H-BPC} proved the result predicted by the probabilistic model. Unconditionally, the best current results are an asymptotic formula~\eqref{eq:pntalmostallshort} for almost all $y \in [X, 2X]$ when $H\geq X^{1/6+o(1)}$ (see e.g.~\cite[Theorem 9.1]{harman}), and a lower bound when $H\geq X^{1/20}$ (see Jia \cite{jia}).

Here we prove the asymptotic formula (\ref{eq:pntalmostallshort}) for almost all $y$ with $H$ in the range predicted by the probabilistic model under the unlikely assumption of existence of Siegel zeros.

\begin{corollary}
\label{cor:shortsexzero}
Let $C \geq 2$ and $\varepsilon > 0$. Let $\chi$ be a primitive quadratic character modulo $q \geq 2$ and assume that $L(s,\chi)$ has a real zero $\beta_0$ such that
\[
\beta_0 = 1-\frac{1}{\eta \log q}.
\]
for some $\eta \geq 10$.

Let $X \geq q^{10},$ write $V := \frac{\log X}{\log q} \geq 10$, and let $2 \leq H \leq X^{1/3}$. Then
\[
\begin{split}
&\int_X^{2X} \left(\sum_{y < n \leq y+H} \Lambda(n) - H \right)^2 dy  \\
&\ll_{C, \varepsilon} H X \log X + H^2 X \left(\exp(-C \sqrt{V \log \eta}) + \exp(-C(\log X)^{3/5-\varepsilon}) + \frac{V \log^6 \eta}{\eta} \right).
\end{split}
\]
\end{corollary}
This implies that as soon as 
\[
\eta \to \infty, \quad \frac{H}{\log X} \to \infty, \quad \text{and} \quad  q^{10} \leq X \leq q^{\eta^{1-\delta}},
\] 
we get the asymptotic formula
\[
\sum_{y < p \leq y+H} 1 = (1+o(1))\frac{H}{\log y}
\]
for almost all $y \in [X, 2X]$.

\begin{remark}
Under the assumption of the existence of Siegel zeros one can show that the Pair correlation conjecture~\cite{montgomery} fails drastically, that is, almost all zeros of $\zeta(s)$ in some range depending on $q$ lie on a lattice with normalized distances in $(\frac{1}{2}+o(1)) \mathbb{Z}$ (the so-called alternative hypothesis, this follows e.g. from \cite[Proposition 9.2]{Conrey-Iwaniec} assuming $\eta=\log^{100} q$ for zeros of height $T=\exp(\log^{10} q)$). Note however that in~\cite{H-BPC}, instead of the full Pair correlation conjecture, Heath-Brown only requires a weaker upper bound for Montgomery's \cite{montgomery} function $F$, which is consistent with the alternative hypothesis.
\end{remark}

\begin{remark}
All our results hold also if instead of existence of Siegel zeros we assume that $L(s,\chi)$ takes a small value at $s=1$, that is, assuming
\[
L(1,\chi) \leq \frac{1}{\eta \log q}.
\]
Indeed, by \cite[Theorem 11.4]{MVBook} (using~\cite[(11.7)]{MVBook} in the contrapositive direction) this implies that $L(s,\chi)$ has a real zero $\beta_0$ with
\[
1-\beta_0 \ll L(1,\chi) \ll  \frac{1}{\eta \log q}.
\]

In the opposite direction the best current result loses essentially two factors of $\log q$, that is, we have $L(1,\chi) \ll (1-\beta_0)(\log^2 q)/\log \log q$ by work of Friedlander and Iwaniec~\cite{F-ISiegel}. For this reason we state our results in terms of Siegel zeros. 

Note that in many articles such as \cite{Illusory,JoriChar} one has to assume that $\chi$ is exceptional in a very strong sense, that is, $L(1,\chi) \ll \log^{-C} q$ for some large constant $C.$ Combining the ideas from this paper with the argument in \cite{JoriChar} it is possible replace this strong assumption by (\ref{eq:szeroassumption}) with a similar dependency on $\eta$ as in Corollary \ref{cor:simpleMT}. 
\end{remark}

Let us here briefly discuss why having a Siegel zero is a helpful assumption; we will discuss the proof strategy rigorously and in more detail in Section~\ref{se:initial}.

If $L(s, \chi)$ has a zero at $s = \beta_0$ with $\beta_0$ close to $1$, then
\[
L(1,\chi)^{-1} = \sum_{n} \frac{\mu(n)\chi(n)}{n} = \prod_p \left(1+\frac{\mu(p)\chi(p)}{p} \right)
\]
is large, so that $\chi(p)=\mu(p)$  for most primes (in a wide range depending on $q$) and heuristically we have
\begin{equation}
\label{eq:LambdaApprox}
\Lambda = \mu \ast \log  \approx \chi \ast \log,
\end{equation}
so that we can hope to replace $\Lambda$ by a $\chi \ast \log$. But the function $\chi \ast \log$ is of similar complexity as the divisor function $\tau=1\ast 1$ when the modulus $q$ is small compared to $x$.

Hence, in order to prove our theorems, we need to show that the contribution of the error in the approximation~\eqref{eq:LambdaApprox} is small as well as study 
\[
\sum_{n \leq X} (\chi \ast \log)(n)(\chi \ast \log)(\pm n+h)
\]
which has similar complexity as divisor correlations. Friedlander and Iwaniec \cite{Illusory} have shown that the error in~\eqref{eq:LambdaApprox} can be controlled if we can solve the corresponding ternary divisor problem. In our case we cannot, but we can still deal with the error, once we restrict both sides to numbers without large prime factors. In general, for sequences with relatively large (logarithmic) density one can exploit crude bounds to get a result without knowledge of the corresponding ternary divisor problem.

\subsection*{Notation} 
\label{sse:Notation}
We write $\mathbf{1}_P$ for the indicator function of the claim $P$. We write $\Lambda(n)$, $\mu(n)$, $\varphi(n)$, $\tau(n)$ for the von Mangoldt function, M\"obius function, Euler $\varphi$-function and the divisor function. These functions are understood to equal $0$ for non-positive integers. For arithmetic functions $f, g \colon \mathbb{N} \to \mathbb{C}$ we define the Dirichlet convolution 
\[
(f \ast g)(n) := \sum_{n = km} f(k)g(m).
\]

For $f \colon \mathbb{R} \to \mathbb{C}$ and $g\colon \mathbb{R} \to \mathbb{R}^+$, we write $f(x) = O(g(x))$ or $f(x) \ll g(x)$ if there exists a constant $C > 0$ such that $|f(x)| \leq C g(x)$ for every $x$. Furthermore, for positive valued $f$ and $g$ we write $f(x) \asymp g(x)$ when $g(x) \ll f(x) \ll g(x)$.  If there is a subscript (e.g. $O_k(g(x))$), then the implied constant is allowed to depend on the parameter(s) in the subscript.

We say that a function $g \colon \mathbb{R} \to \mathbb{R}$ is smooth if it has derivatives of all orders. 

For $u \in \mathbb{C}$ we write $e(u) := e(2\pi i u)$ and for $q \in \mathbb{N}$ we write $e_q(u) = e(u/q)$. For any function
\[
g \in \mathbb{L}^1(\mathbb{R}) := \left\{f \colon \mathbb{R} \to \mathbb{C} \colon \int_{-\infty}^\infty |f(x)| dx < \infty\right\},
\]
we denote by $\widehat{g}$ the Fourier transform
\[
\widehat{g}(\xi) = \int_{-\infty}^\infty g(x) e(-\xi x) dx.
\]

For $a \in \mathbb{Z}$ and $q \in \mathbb{N}$ we write $\overline{a}$ for the inverse of $a \pmod{q}$ (the modulus will be clear from the context, e.g. in $e(\frac{c\overline{u}}{v})$ the inverse is $\pmod{v}$).

\section{Initial steps} 
\label{se:initial}
We start by replacing the sums in Theorems \ref{th:MTzero} and~\ref{th:Goldbach} by smoothed variants. Let $\delta= X^{-\eps}$ for some small $\eps >0$ and let $g:\mathbb{R} \to [0,1]$ be a smooth function that is supported on $[1,2]$ and equals $1$ on $[1+\delta,2-\delta]$. Assume further that the derivatives of $g$ satisfy $g^{(j)}(x) \ll \delta^{-j}$ for every $x$.

In case of Theorem~\ref{th:MTzero} we first decompose the summation condition $n \leq X$ dyadically into conditions $n \in (x, 2x]$ with $x \leq X/2$. We estimate the contribution of $x \leq X^{1-\varepsilon/4}$ trivially, and for the remaining $x$ we replace the condition $\mathbf{1}_{n \in (x, 2x]}$  by $g(n/x)$ with an error term $O(\delta x \log^2 X) = O(X^{1-\eps/2})$. Thus it suffices to show that
\[
\begin{split}
&\sum_{n} g\left(\frac{n}{X}\right)\Lambda(n) \Lambda(n+h)= \int g\left(\frac{y}{X}\right) dy \cdot \mathfrak{S}_h\left(1 + \mathbf{1}_{\varphi(2^r) \mid h} (-1)^{\frac{h}{\varphi(2^r)}} \prod_{\substack{p \mid q' \\ p \nmid h}} \frac{-1}{p-2}\right) \\
&\qquad + O_{C, \varepsilon}\left(\frac{h}{\varphi(h)} X \left(\exp(-C \sqrt{V \log \eta}) + \exp(-C(\log X)^{3/5-\varepsilon}) + \frac{V (\log \eta)^6}{\eta} \right)\right).
\end{split}
\]
whenever $0 < h \leq X^{1+\varepsilon/2}$ (we can assume that $h$ is positive by symmetry). 

In case of Theorem~\ref{th:Goldbach}, note that by symmetry we can concentrate on the case $n_1 < h/2 < n_2$  and that when $n_1 \leq h^{1-\varepsilon/4}$ we can use a trivial estimate. Arguing as above with similar $g$, we see that it suffices to show that
\[
\begin{split}
&\sum_{n} g\left(\frac{n}{X}\right)\Lambda(n) \Lambda(h-n) = \int g\left(\frac{y}{X}\right) dy \cdot \mathfrak{S}_h\left(1 + \chi(-1)\mathbf{1}_{\varphi(2^r) \mid h} (-1)^{\frac{h}{\varphi(2^r)}} \prod_{\substack{p \mid q' \\ p \nmid h}} \frac{-1}{p-2}\right) \\
&\qquad + O_C\left(\frac{h}{\varphi(h)} X \left(\exp(-C V \sqrt{\log \eta}) + \exp(-C(\log h)^{3/5-\varepsilon}) + \frac{V (\log \eta)^6}{\eta} \right)\right). 
\end{split}
\]
for every $X \in [h^{1-\varepsilon/3}, h/4]$.

We shall deal with Theorems~\ref{th:MTzero} and~\ref{th:Goldbach} simultaneously, and thus consider
\[
\sum_{n} g\left(\frac{n}{X}\right)\Lambda(n) \Lambda(\pm n+h).
\]
Note that when $n$ is in the support of $g(n/X)$ with $X$ as above, we have $\pm n + h \geq \max\{X, h/2\}.$

Let $\chi$ be a quadratic character modulo $q$. Following~\cite{Illusory, JoriChar}, we write
\begin{equation}
\label{eq:deflamlam'}
\lambda := 1 \ast \chi \quad \text{and} \quad \lambda' := \chi \ast \log, 
\end{equation}
so that
\[
\lambda \ast \Lambda = (1 \ast \chi) \ast (\mu \ast \log) = (1 \ast \mu) \ast (\chi \ast \log) = \lambda'.
\]
Note also that $\lambda(n) \geq 0$ and $\lambda'(n) \geq \Lambda(n) \geq 0$ (since $\lambda' = \lambda \ast \Lambda$). 
By above
\begin{equation}
\label{eq:lambda'-Lambda1}
\lambda'(n) = (\lambda \ast \Lambda)(n) =  \Lambda(n) + \sum_{\substack{n = km \\ m > 1}} \Lambda(k) \lambda(m),
\end{equation}
so we have obtained a formula for the error term in the approximation~\eqref{eq:LambdaApprox}. Similarly as in \cite{JoriChar}, we now restrict this approximation to rough numbers to ensure that $m$ is large. For large $m$ we expect by the assumption of existence of Siegel zeros  that the function $\lambda(m)$ is lacunary, which will make the sum in~\eqref{eq:lambda'-Lambda1} small.

For $w \geq 2$ and $k \in \mathbb{N}$, write
\[
P(w) = \prod_{p < w} p \quad \text{and} \quad P_k(w) = \prod_{\substack{p < w \\ p \nmid k}} p.
\]
Let $u$ be a large parameter to be chosen later (see \eqref{eq:uchoice}) and write $z := X^{1/u}$. Adding  the condition $(n, qP(z)) = 1$ to both sides of~\eqref{eq:lambda'-Lambda1}, we get 
\begin{equation}
\label{eq:lambda'-Lambda}
\Lambda(n) = \lambda'(n)\mathbf{1}_{(n, q P(z)) = 1} - \sum_{\substack{n = km \\ m \geq z \\ (km, q P(z)) = 1}} \Lambda(k) \lambda(m) + O\left(\log n \cdot \mathbf{1}_{\substack{n = p^\nu \\ p \mid qP(z)}}\right).
\end{equation}
We define
\begin{equation}
\label{eq:cndef}
c_n := \sum_{\substack{n = km \\ m \geq z \\ (km, P(z)) = 1}} \Lambda(k) \lambda(m). 
\end{equation}
Note that $0 \leq \lambda'(n) \leq \tau(n) \log n$ and $0 \leq c_n \leq \mathbf{1}_{(n, P(z)) = 1} \tau(n)^2 \log n$. When $g \colon \mathbb{R} \to [0,1]$ is a smooth function supported on $[1, 2]$, we obtain
\begin{equation}
\label{eq:Lnn+hSplit}
\begin{split}
&\sum_{n} g\left(\frac{n}{X}\right) \Lambda(n) \Lambda(\pm n+h) = \sum_{\substack{n \\ (n(\pm n+h), q P(z)) = 1}}  g\left(\frac{n}{X}\right) \lambda'(n) \lambda'(\pm n+h) \\
& \quad + O\left((z + \omega(q)) \log^2 X + \log X \sum_{\substack{n \leq 2X \\ (n(\pm n+h), q P(z)) = 1}} (c_n \tau(\pm n+h)^2 + \tau(n) c_{\pm n+h})\right).
\end{split}
\end{equation}
We deal with the error term using the following lemma which will quickly follow from Henriot's bound on correlations of multiplicative functions (see Section~\ref{ssec:multfunct} for the proof).
\begin{lemma}
\label{le:cncor}
Let $c_n$ be as in~\eqref{eq:cndef}, let $X \geq 3, u \geq 2$, and $z = X^{1/u}$. Then, for any $0 < |h| \leq X^{10}$, we have
\[
\begin{split}
\sum_{\substack{n \leq 4X \\ (n(\pm n+h), P(z)) = 1}} c_n \tau(\pm n+h)^2 \ll \frac{h}{\varphi(h)} \frac{X}{\log X} \left(u^4 \sum_{\substack{z \leq m \leq 4X/z \\ (m,P(z))=1}} \frac{\lambda(m)}{m} + \frac{u^{10}}{z}\right).
\end{split}
\]
\end{lemma}

Here the sum over $m$ can be estimated in terms of $\eta$. The following lemma will be proved in Section~\ref{se:ExcCharCons}.
\begin{lemma} \label{le:exzero}
Let $\chi$ be a primitive quadratic character modulo $q \geq 2$. Assume that $L(s,\chi)$ has a real zero $\beta_0$ such that
\[
\beta_0 = 1-\frac{1}{\eta \log q}
\]
for some $\eta \geq 10$. Let $z=q^{v}$ for some $v \in \mathbb{R}_+$. Then for any $Y > z$
\[
\sum_{\substack{z \leq m \leq Y \\ (m, P(z)) = 1}} \frac{\lambda(m)}{m} \ll \left(\frac{1}{v^{2}\eta^{v/2}} + \frac{v}{\eta} \cdot \frac{\log Y}{\log z} + \frac{1}{z}\right)  \bigg(\frac{\log Y}{\log z}\bigg)^2.
\] 
\end{lemma}

To deal with the main term in~\eqref{eq:Lnn+hSplit} we use the following proposition which we will prove in Section~\ref{se:typeI2prop}. Note that this proposition is unconditional and gives an asymptotic formula for generalized divisor function correlations over rough numbers. The parameter $\beta$ refers to the $\beta$-sieve described in Lemma \ref{le:Sieve}. 
\begin{proposition}
\label{prop:typeI2}
Let $\delta > 0$. Let $X, M_1, M_2, N_1, N_2 \geq 1$ be such that 
\[
M_1 N_1 \asymp X, \quad X \ll M_2 N_2 \ll \delta^{-1} X, \quad \text{and} \quad M_j \ll N_j. 
\]
Let $h \in \mathbb{Z}_+$. Let $q \geq 1$ and let $\chi_1, \chi_2, \psi_1, \psi_2$ be real characters $\pmod{q}$. Let $f \colon \mathbb{R}_+^4 \to [0, 1]$ be smooth and compactly supported and suppose that for all $j \geq 0$ and $v \in \{1, 2\}$,
\[
\frac{\partial^j}{\partial x_v^j}  f(x_1,x_2,y_1,y_2), \quad \frac{\partial^j}{\partial y_v^j}  f(x_1,x_2,y_1,y_2) \ll_j \delta^{-j}.
\] 
Let $A \in \mathbb{N}$. Assume that $\beta$ is sufficiently large in terms of $A$ and $u \geq 1000\beta$. Write $z = X^{1/u}$ and, for $r \geq 0$,
\begin{equation}
\label{eq:defzr}
z_r := z^{\left((\beta-1)/\beta\right)^{r}}
\end{equation}
Then 
\[
\begin{split}
&\sum_{\substack{m_1, m_2, n_1, n_2 \\ \pm m_1 n_1 + h = m_2 n_2 \\ (m_1 m_2 n_1 n_2, P(z)) = 1}} \chi_1(m_1) \chi_2(m_2) \psi_1(n_1) \psi_2(n_2) f\left(\frac{m_1}{M_1}, \frac{m_2}{M_2}, \frac{n_1}{N_1}, \frac{n_2}{N_2}\right) \\
&= \prod_{\substack{p < z \\ p \mid h, p \nmid q}} \left(1-\frac{1}{p}\right) \prod_{\substack{p < z \\ p \nmid hq}} \left(1-\frac{2}{p}\right) \cdot \frac{1}{q} \sum_{\gamma \Mod{q}} \psi_1(\gamma) \psi_2(\pm\gamma +h)  \\
& \cdot \sum_{\substack{m_1, m_2 \\ (m_1 m_2, P(z)) = 1}} \frac{\chi_1(m_1) \chi_2(m_2) \psi_1(m_1) \psi_2(m_2)}{m_1 m_2} \int f\left(\frac{m_1}{M_1}, \frac{m_2}{M_2}, \frac{y}{m_1 N_1}, \frac{\pm y + h}{m_2 N_2}\right) dy \\
& +O\Biggl(\sum_{\substack{r_1, r_2, r_3 \geq 0 \\ \max r_j \geq \frac{u}{1000}-\beta}} 2^{-A(r_1 +r_2 +r_3)} \sum_{\substack{m_1, m_2, n_1, n_2 \\ \pm m_1 n_1 +h = m_2 n_2 \\  (n_1, P(z_{r_1})) = (n_2, P_h(z_{r_2}))= 1 \\ (m_1, P(z_{r_3})) = (m_2, h P(z)) = 1}} (\tau(m_1) \tau(n_1) \tau(n_2))^{A+1} \\
& \quad \cdot f\left(\frac{m_1}{M_1}, \frac{m_2}{M_2}, \frac{n_1}{N_1}, \frac{n_2}{N_2}\right) \Biggr) + O_A\left( \delta^{-3} X^{7/9} q^2 + \frac{h}{\varphi(h)} \frac{X}{\log^2 X} \left(\frac{u^6}{z} + e^{-Au/3000}\right)\right).
\end{split}
\]
\end{proposition}

An application of Henriot's bound (see Lemma~\ref{le:Henriot} below) will allow us to show in Section~\ref{ssec:multfunct} that 
\begin{equation}
\label{eq:ConsHenriot}
\begin{split}
&\sum_{\substack{r_1, r_2, r_3 \geq 0 \\ \max r_j \geq \frac{u}{1000}-\beta}} 2^{-A(r_1 + r_2+r_3)} \sum_{\substack{m_1, m_2, n_1, n_2 \\ m_1 n_1 \leq 10X \\ \pm m_1 n_1 +h = m_2 n_2 \\  (n_1, P(z_{r_1})) = (n_2, P_h(z_{r_2}))= 1 \\ (m_1, P(z_{r_3})) = (m_2, hP(z)) = 1}} (\tau(m_1) \tau(n_1) \tau(n_2))^{A+1} \\
&\ll_A \frac{h}{\varphi(h)} \frac{X}{\log^2 X} e^{-A u/2000}.
\end{split}
\end{equation}

In Section~\ref{se:chisums} we will prove the following lemma which helps us in evaluating the main term we obtain from Proposition~\ref{prop:typeI2}.
\begin{lemma} 
\label{le:chisums}
Let $\chi$ be a primitive quadratic character modulo $q \geq 2$. Assume that $L(s,\chi)$ has a real zero $\beta_0$ such that
\[
\beta_0 = 1-\frac{1}{\eta \log q}
\]
for some $\eta \geq 10$. Let $X \geq 3, u \geq 2$, and $z=X^{1/u}=q^v$. Assume that $X^{1/10} \leq N \leq X^2$ and $1 \leq y \leq X^2$. Then, for any $\varepsilon > 0$ and $A, C \geq 2$, 
\[
\sum_{\substack{n \leq N \\ (n,P(z))=1}} \frac{\chi(n)\log (y/n)}{n} =  (1+ O_{A,C,\varepsilon}(\mathcal{E}))\prod_{p < z}  \left(1-\frac{1}{p}\right)^{-1} 
\]
with
\[
\mathcal{E} =\frac{u^4}{v^2 \eta^{v/2}} + \frac{vu^5}{\eta}  +\frac{u^4}{z} + e^{-A u/3000} +  \exp(- C\log^{3/5-\eps} X).
\]
\end{lemma}

The complete character sums resulting from Proposition~\ref{prop:typeI2} will be evaluated using the following elementary lemma which will be proved in Section~\ref{ssec:char}.
\begin{lemma}
\label{le:Chib2}
Let $q \geq 2$ and let $\chi$ be a primitive quadratic character of modulus $q = 2^r q'$ with $r \geq 0$ and $2 \nmid q'$. Let $\chi_0$ be the principal character $\pmod{q}$, and let $h$ be an even integer. Then
\begin{align}
\label{eq:chi00}
\frac{1}{q}\sum_{m \Mod{q}} \chi_0(m)\chi_0(\pm m+h) &= \prod_{p \mid (q, h)} \left(1-\frac{1}{p}\right)\prod_{\substack{p \mid q \\ p \nmid h}} \left(1-\frac{2}{p}\right), \\
\label{eq:chip0}
\frac{1}{q} \sum_{m \Mod{q}} \chi(m)\chi_0(\pm m+h) &= \frac{-\chi(\mp h)}{q}, \\
\label{eq:chipp}
\frac{1}{q}\sum_{m \Mod{q}} \chi(m) \chi(\pm m+h) &=  \mathbf{1}_{\varphi(2^r) \mid h} (-1)^{\frac{h}{\varphi(2^r)}} \chi(\pm 1) \prod_{p \mid (q, h)} \left(1-\frac{1}{p}\right)\prod_{\substack{p \mid q \\ p \nmid h}} \frac{-1}{p}.
\end{align}
\end{lemma}
Precise evaluation of the character sum~\eqref{eq:chipp} is, in addition to keeping track of some dependencies, the reason we can deal with the case $(h, q)$ is close to $q$ in Theorems~\ref{th:MTzero} and~\ref{th:Goldbach}. As Tao and Ter\"av\"ainen~\cite{TT} consider the more general problem~\eqref{eq:TTgeneral}, they face more complicated character sums and need to apply Weil's bound, and thus they do not obtain similar uniformity in $h$.

In Section~\ref{se:twinPrimeThms} we shall use the results stated here to prove Theorems~\ref{th:MTzero} and~\ref{th:Goldbach}. Then in Section~\ref{se:ShortThms} we deduce Corollary~\ref{cor:shortsexzero} from Theorem~\ref{th:MTzero}.

\begin{remark}
Before moving on, let us briefly discuss how our approach differs from that of Tao and Ter\"av\"ainen~\cite{TT}. We establish Proposition~\ref{prop:typeI2} using the $\beta$-sieve whereas Tao and Ter\"av\"ainen~\cite{TT} use Selberg's sieve in their arguments. Using the $\beta$-sieve makes replacing the condition $\mathbf{1}_{(n, P(z)) = 1}$ by a type I sum more transparent since the remainder is easier to analyse, see Lemma~\ref{le:Sieve}(i) for a convenient bound. This difference of the two sieves is discussed also in~\cite[Section 10.2]{Opera}. Thanks to using the $\beta$-sieve our arguments are considerably simpler than those in~\cite{TT}, and we get improved error terms automatically.

More precisely, the starting point of~\cite{TT} is to replace $\Lambda(n)$ by $(\chi \ast \log)(n) \nu(n),$ where $\nu(n)$ are Selberg sieve weights. Replacing $\Lambda(n)$ by $\Lambda(n) \nu(n)$ is almost immediate thanks to the support of $\Lambda(n)$. However, replacing $\Lambda(n) \nu(n)$ by $(\chi \ast \log)(n) \nu(n)$ is somewhat more complicated and this leads to weaker error terms. Also later the Selberg sieve coefficients complicate matters in~\cite{TT}. Despite this, it might be possible to argue more carefully with the Selberg sieve and obtain error terms comparable to ours.  
\end{remark}

\section{Lemmas} \label{se:lemmas}
\subsection{Multiplicative functions}
\label{ssec:multfunct}
We will use some standard estimates for multiplicative functions. Note first that, when $f \colon \mathbb{N} \to \mathbb{C}$ is a divisor-bounded (i.e. $f(n) \leq \tau(n)^A$ for some $A \geq 1$) multiplicative function, we have
\begin{equation}
\label{eq:f(n)/n_average}
\sum_{n \leq X} \frac{|f(n)|}{n} \ll \prod_{p \leq X} \left(1+\frac{|f(p)|}{p}\right).
\end{equation}
Furthermore, by Mertens' theorem
\begin{equation}
\label{eq:Mertens}
\prod_{w < p \leq z} \left(1+\frac{k}{p}\right) \asymp \left(\frac{\log z}{\log w}\right)^k.
\end{equation}
We shall also need the following consequences of Henriot's bound~\cite{Henriot1, Henriot2}.

\begin{lemma}
\label{le:Henriot}
Let $X \geq 2$ and $2 \leq w_1, w_2 \leq X$. Let also $1 \leq |h| \leq X^{10}$.
\begin{enumerate}[(i)]
\item Let $m_1, m_2 \geq 0$. Then
\[
\begin{split}
&\sum_{n \leq X} \tau(n)^{m_1} \mathbf{1}_{(n, P(w_1)) = 1} \tau(\pm n+h)^{m_2} \mathbf{1}_{(\pm n+h, P_h(w_2)) = 1} \\
&\ll_{m_1, m_2} \frac{h}{\varphi(h)} \cdot \frac{X}{\log^2 X} \left(\frac{\log X}{\log w_1}\right)^{2^{m_1}+1} \left(\frac{\log X}{\log w_2}\right)^{2^{m_2}}.
\end{split}
\]
\item 
Let $m \in \mathbb{Z}$ be such that $(m, h P(w_2)) = 1$. Then
\[
\begin{split}
&\sum_{n \leq X} \mathbf{1}_{(n, P(w_1)) = 1} \mathbf{1}_{(\pm mn+h, P(w_2)) = 1} \tau(\pm mn+h)^2 \ll \frac{h}{\varphi(h)} \frac{X}{\log^2 X} \frac{\log X}{\log w_1} \left(\frac{\log X}{\log w_2}\right)^{4}.
\end{split}
\]
\end{enumerate}
\end{lemma}
\begin{proof}
\textbf{Proof of (i).}  We apply Henriot's bound~\cite[Theorem 3]{Henriot1} with
\[
\begin{split}
Q_1(n) &= n, \, Q_2(n) = \pm n+h, \, D = h^2, \\
F(n_1, n_2) &= \mathbf{1}_{(n_1, P(w_1)) = 1} \mathbf{1}_{(n_2, P_h(w_2)) = 1} \tau(n_1)^{m_1} \tau(n_2)^{m_2}, \\
\rho_{Q_1}(n) &= |\{v \pmod{n} \colon v \equiv 0 \Mod{n}\}| = 1, \\
\rho_{Q_2}(n) &= |\{v \pmod{n} \colon \pm v+h \equiv 0 \Mod{n}\}| = 1 \\
\rho(p) &=  |\{v \pmod{p} \colon v(\pm v+h) \equiv 0 \Mod{p}\}| =
\begin{cases}
2 & \text{if $p \nmid h$} \\
1 & \text{otherwise},
\end{cases} 
\\
\Delta_D &= \prod_{p \mid h} \Biggl(1+F(p, 1) \frac{|\{n \Mod{p^2} \colon p \Vert n, p \nmid \pm n+h\}|}{p^2} + F(1, p) \frac{|\{n \Mod{p^2} \colon p \nmid n, p \Vert \pm n+h\}|}{p^2} \\
& \qquad + F(p, p) \frac{|\{n \Mod{p^2} \colon p \Vert n, p \Vert \pm n+h\}|}{p^2}\Biggr) \ll \prod_{\substack{p \mid h \\ p \geq w_1}} \left(1+\frac{2^{m_1 + m_2}}{p}\right).
\end{split}
\]
Applying~\cite[Theorem 3]{Henriot1} and then~\eqref{eq:f(n)/n_average}, we obtain
\[
\begin{split}
&\sum_{n \leq X} \tau(n)^{m_1}\mathbf{1}_{(n, P(w_1)) = 1} \tau(\pm n+h)^{m_2} \mathbf{1}_{(\pm n+h, P_h(w_2)) = 1} \\ 
&\ll \Delta_D X \prod_{p \leq X} \left(1-\frac{\rho(p)}{p}\right) \sum_{\substack{n_1 n_2 \leq X \\ (n_1 n_2, D) = 1}} F(n_1, n_2) \frac{\rho_{Q_1}(n) \rho_{Q_2}(n)}{n_1 n_2} \\
&\ll X \prod_{\substack{p \mid h \\ p \geq w_1}} \left(1+\frac{2^{m_1 + m_2}}{p}\right) \cdot \prod_{p \leq X} \left(1-\frac{2}{p}\right) \prod_{p \mid h} \left(1+\frac{1}{p}\right) \\
& \qquad \cdot \prod_{\substack{w_1 \leq p \leq X \\ p \nmid h}} \left(1+\frac{2^{m_1}}{p}\right) \prod_{\substack{w_2 \leq p \leq X \\ p \nmid h}} \left(1+\frac{2^{m_2}}{p}\right).
\end{split}
\]
Here
\[
\begin{split}
\prod_{\substack{p \mid h \\ p \geq w_1}} \left(1+\frac{2^{m_1 + m_2}}{p}\right) &\ll \exp\left(2^{m_1 + m_2} \sum_{\substack{p \mid h \\ p \geq w_1}} \frac{1}{p}\right) \ll \exp\left(2^{m_1 + m_2} \sum_{\substack{p \leq 10\frac{\log X}{\log w_1}}} \frac{1}{p}\right) \\
&\ll_{m_1, m_2} \left(\log \frac{\log X}{\log w_1}\right)^{2^{m_1 + m_2}} \ll_{m_1, m_2} \frac{\log X}{\log w_1}.
\end{split}
\]
and (i) follows from~\eqref{eq:Mertens}.

\textbf{Proof of (ii).} We use a similar application of Henriot's bound~\cite[Theorem 3]{Henriot1} --- this time
\[
\begin{split}
Q_1(n) &= n, \, Q_2(n) = \pm mn+h, \, D = h^2, \, F(n_1, n_2) = \mathbf{1}_{(n_1, P(w_1)) = 1} \mathbf{1}_{(n_2, P(w_2)) = 1} \tau(n_2)^2, \\
\rho_{Q_1}(n) &= |\{v \Mod{n} \colon v \equiv 0 \Mod{n}\}| = 1 \\
\rho_{Q_2}(n) &= |\{v \pmod{n} \colon \pm vm+h \equiv 0 \Mod{n}\}| = 
\begin{cases}
1 & \text{if $(n, m) = 1$} \\
0 & \text{otherwise}.
\end{cases} \\
\rho(p) &=  |\{v \pmod{p} \colon v(\pm vm+h) \equiv 0 \Mod{p}\}| =
\begin{cases}
2 & \text{if $p \nmid hm$} \\
1 & \text{otherwise},
\end{cases}
\end{split}
\]
and 
\[
\begin{split}
\Delta_D &= \prod_{p \mid h} \Biggl(1+ F(p, p) \frac{|\{n \Mod{p^2} \colon p \Vert n, p \Vert \pm mn+h\}|}{p^2}\Biggr) \ll \prod_{\substack{p \mid h \\ p \geq w_2}} \left(1+\frac{4}{p}\right).
\end{split}
\]
Applying~\cite[Theorem 3]{Henriot1} and then~\eqref{eq:f(n)/n_average}, we obtain
\[
\begin{split}
&\sum_{n \leq X} \mathbf{1}_{(n, P(w_1)) = 1} \mathbf{1}_{(\pm mn+h, P(w_2)) = 1} \tau(\pm mn+h)^2 \\
&\ll \Delta_D X \prod_{p \leq X} \left(1-\frac{\rho(p)}{p}\right) \sum_{\substack{n_1 n_2 \leq X \\ (n_1 n_2, D) = 1}} F(n_1, n_2) \frac{\rho_{Q_1}(n) \rho_{Q_2}(n)}{n_1 n_2} \\
&\ll X \prod_{\substack{p \mid h \\ p \geq w_2}} \left(1+\frac{4}{p}\right) \cdot \prod_{\substack{p \leq X}} \left(1-\frac{2}{p}\right) \prod_{\substack{p \leq X \\ p \mid hm}} \left(1+\frac{1}{p}\right) \cdot \prod_{\substack{w_1 \leq p \leq X \\ p \nmid h}} \left(1+\frac{1}{p}\right) \prod_{\substack{w_2 \leq p \leq X \\ p \nmid hm}} \left(1+\frac{2^2}{p}\right)
\end{split}
\]
Since $(m, hP(w_2)) = 1$, this is 
\[
\ll \frac{h}{\varphi(h)} \frac{X}{\log^2 X} \frac{\log X}{\log w_1} \left(\frac{\log X}{\log w_2}\right)^{4} \prod_{\substack{w_2 \leq p \leq X \\ p \mid m}} \left(1+\frac{1}{p}-\frac{4}{p}\right)
\]
and the claim (ii) follows.
\end{proof}

Now Lemma~\ref{le:cncor} and~\eqref{eq:ConsHenriot} follow quickly:
\begin{proof}[Proof of Lemma~\ref{le:cncor}]
Consider
\begin{equation}
\label{eq:cntau}
\begin{split}
&\sum_{n \leq 4X} c_{n} \tau(\pm n+h)^2 \mathbf{1}_{(\pm n+h, P(z)) = 1}
\end{split}
\end{equation}
Using $c_n \leq \mathbf{1}_{(n, P(z)) =1} \tau(n)^2 \log n$, we see that those $n$ with $(n, h) > 1$ contribute
\[
\begin{split}
&\ll \log X  \sum_{\substack{n \leq 4X \\ (n, h) > 1}} \tau(n)^2 \mathbf{1}_{(n, P(z)) = 1} \tau(\pm n+h)^2 \mathbf{1}_{(\pm n+h, P(z)) = 1} \\
&\ll \log X \sum_{\substack{p \mid h \\ z \leq p \leq 4X}} \sum_{\substack{n \leq 4X/p}} \tau(n)^2 \mathbf{1}_{(n, P(z)) = 1} \tau(\pm n+h/p)^2 \mathbf{1}_{(\pm n+h/p, P(z)) = 1}.
\end{split}
\]
By Lemma~\ref{le:Henriot}(i) this is 
\[
\ll \frac{h}{\varphi(h)} X \log X \sum_{\substack{p \mid h \\ z \leq p \leq 4X}} \frac{u^9}{p \log^2(8X/p)} \ll  \frac{h}{\varphi(h)} \frac{X}{\log X} \frac{u^{10}}{z}
\]
since $z = X^{1/u} \leq X^{1/2}$.
On the other hand, those $n$ with $(n, h) = 1$ contribute to~\eqref{eq:cntau} by~\eqref{eq:cndef}
\[
\begin{split}
&\ll \sum_{\substack{z \leq m \leq 4X/z \\ (m, h P(z)) = 1}} \lambda(m) \sum_{\substack{z \leq \ell \leq 4X/m  \\ (\pm \ell m + h, P(z)) = 1}} \Lambda(\ell) \tau(\pm \ell m +h)^2 \\
&\ll \log X \sum_{\substack{z \leq m \leq 4X/z \\ (m, hP(z)) = 1}} \lambda(m) \sum_{\substack{z \leq \ell \leq 4X/m  \\ (\pm \ell m + h, P(z)) = 1}} \mathbf{1}_{(\ell, P(x^{1/4})) = 1} \tau(\pm \ell m +h)^2.
\end{split}
\]
The claim follows from applying Lemma~\ref{le:Henriot}(ii) to the inner sum.
\end{proof}

\begin{proof}[Proof of~\eqref{eq:ConsHenriot}]
We have, for any $r_1, r_2, r_3 \geq 0$,
\[
\begin{split}
& \sum_{\substack{m_1, m_2, n_1, n_2 \\ m_1 n_1 \leq 10X \\ \pm m_1 n_1 +h = m_2 n_2 \\  (n_1, P(z_{r_1})) = (n_2, P_h(z_{r_2}))= 1 \\ (m_1, P(z_{r_3})) = (m_2, hP(z)) = 1}} (\tau(m_1) \tau(n_1) \tau(n_2))^{A+1} \\
& \ll \sum_{n \leq 10X} \tau(n)^{2A+3} \tau(\pm n+h)^{A+2} \mathbf{1}_{(n, P(\min\{z_{r_1}, z_{r_3}\})) = 1} \mathbf{1}_{(\pm n+h, P_h(z_{r_2})) = 1}.
\end{split}
\]
By Lemma~\ref{le:Henriot}(i) and recalling the definition of $z_r$ from~\eqref{eq:defzr}, the left hand side of~\eqref{eq:ConsHenriot} is
\[
\begin{split}
&\ll_A \frac{h}{\varphi(h)}\frac{X}{\log^2 X} \sum_{\substack{r_1, r_2, r_3 \geq 0 \\ \max r_j \geq \frac{u}{1000}-\beta}} 2^{-A(r_1+r_2+r_3)} \left(\frac{\log X}{\log \min\{z_{r_1}, z_{r_3}\}}\right)^{2^{2A+3}+1} \left(\frac{\log X}{\log z_{r_2}}\right)^{2^{A+2}} \\
&\ll \frac{h}{\varphi(h)}\frac{X}{\log^2 X} u^{2^{2A+5}} \sum_{\substack{r_1, r_2, r_3 \geq 0 \\ \max r_j \geq \frac{u}{1000}-\beta}} 2^{-A(r_1+r_2+r_3)} \left(\frac{\beta}{\beta-1}\right)^{2^{2A+4} (r_1 + r_2 + r_3)}.
\end{split}
\]
Once $\beta$ is large enough in terms of $A$, this is 
\[
\ll_A \frac{h}{\varphi(h)}\frac{X}{\log^2 X} u^{2^{2A+5}} 2^{-0.9 A u/1000} \ll_A \frac{h}{\varphi(h)} \frac{X}{\log^2 X} e^{-Au/2000}.
\]
\end{proof}

\subsection{Sieves}
The following is a technical version of the fundamental lemma of the sieve. For the standard version, see e.g.~\cite[Lemma 6.11]{Opera}.
\begin{lemma}
\label{le:Sieve}
Let $\theta \in (0, 1/3)$ and $A \in \mathbb{N}$. There exists $\beta_0 = \beta_0(A) \geq 2$ such that the following holds for any $\beta \geq \beta_0$ and $u \geq \beta/\theta$.

Let $X \geq 2$, write $D = X^\theta$ and $z = X^{1/u}$, and define 
\[
z_r := z^{\left((\beta-1)/\beta\right)^{r}}
\]
Then there exist coefficients $\lambda_d$ with $|\lambda_d| \leq 1$ supported on $d \leq D$ such that the following hold.
\begin{enumerate}[(i)]
\item For every $n \in \mathbb{N}$ and any $v \in \mathbb{N}$,
\[
\mathbf{1}_{(n, P_v(z)) = 1} = \sum_{\substack{d \mid (n, P_v(z))}} \lambda_d + O\left( \sum_{r \geq u\theta -\beta} 2^{-Ar} \mathbf{1}_{(n, P_v(z_r)) = 1} \tau(n)^{A+1}\right).
\]
\item If $g \colon \mathbb{N} \to [-1, 1]$ is a multiplicative function for which $|g(p)| \leq 2/p$ for every prime $p$, then
\[
\sum_{d \mid P(z)} \lambda_d g(d) = (1+O_{\beta, A}(e^{-A\theta u/2})) \prod_{p < z} (1-g(p)) 
\]
\end{enumerate}
\end{lemma}
\begin{proof}
Define
\[
\mathcal{D} := \{d = p_1 \dotsm p_r \mid P(z) \colon p_1 > p_2 > \dotsc > p_r, p_1 \dotsm p_m p_m^{\beta} < D \text{ for all odd $m$} \}
\]
and define the upper bound $\beta$-sieve weights $\lambda_d = \mu(d) \mathbf{1}_{d \in \mathcal{D}}$.

\textbf{Proof of (i).} Consider first the case $v = 1$. By the definition of $\lambda_d$, we have (see e.g. \cite[(6.29) with $\mathcal{A} = \{n\}$]{Opera}),
\[
\mathbf{1}_{(n, P(z)) = 1} = \sum_{d \mid (n, P(z))} \lambda_d - \sum_{r \text{ odd}} S_r(n),
\]
where
\[
S_r(n) := \sum_{\substack{n = p_1 \dotsm p_r k \\ p_r < p_{r-1} < \dotsc < p_1 < z \\ p_1 p_2 \dotsm p_r p_r^{\beta} \geq D \\ p_1 \dotsm p_h p_h^{\beta} < D \text{ for all odd $h < r$}}} \mathbf{1}_{(k, P(p_r)) = 1} 
\]
Since $p_j < z$ and $p_1 p_2 \dotsm p_r p_r^{\beta} \geq D$, the sum in $S_r(n)$ is non-empty only if $(r+\beta)/u \geq \theta \iff r \geq u\theta -\beta$. 

Furthermore, since $\frac{\log D}{\log z} = u\theta \geq \beta$, one can easily show that, for every $r$, in the sum defining $S_r(n)$ one has $p_r \geq z_r$ (see e.g.~\cite[Section 6.3]{IwKo}). We write $m = p_1 \dotsm p_r$ so that obviously $2^{A(\omega(m) - r)} \geq 1$. Hence
\[
S_r(n) \leq \sum_{\substack{n = m k \\ p \mid mk \implies p \geq z_r}} 2^{A\omega(m)-Ar} \leq 2^{-Ar} \tau(n)^{A+1} \mathbf{1}_{(n, P(z_r)) = 1},
\]
and (i) follows in case $v = 1$.

When $v > 1$, we write $n = n' v'$ with $(n', v) = 1$ and $p \mid v' \implies p \mid v$. Then $(n, P_v(w)) = (n', P(w))$ for every $w \geq 2$. Hence, by the case $v = 1$, we obtain
\[
\begin{split}
\mathbf{1}_{(n, P_v(z)) = 1} = \mathbf{1}_{(n', P(z)) = 1} &= \sum_{\substack{d \mid (n', P(z))}} \lambda_d + O\left( \sum_{r \geq u\theta -\beta} 2^{-Ar} \mathbf{1}_{(n', P(z_r)) = 1} \tau(n')^{A+1}\right)\\
& = \sum_{\substack{d \mid (n, P_v(z))}} \lambda_d + O\left( \sum_{r \geq u\theta -\beta} 2^{-Ar} \mathbf{1}_{(n, P_v(z_r)) = 1} \tau(n)^{A+1}\right)
\end{split}
\]
as claimed.

\textbf{Proof of (ii).} By the definition of $\lambda_d$ we have (see e.g.~\cite[(6.31)]{Opera})
\begin{equation}
\label{eq:lambdadgd}
\sum_{d \mid P(z)} \lambda_d g(d)  = \prod_{p < z} (1-g(p)) - \sum_{r \text{ odd}} V_r(z),
\end{equation}
where
\[
V_r(z) := \sum_{\substack{p_r < p_{r-1} < \dotsc < p_1 < z \\ p_1 p_2 \dotsm p_r p_r^{\beta} \geq D \\ p_1 \dotsm p_h p_h^{\beta} < D \text{ for all odd $h < r$}}} g(p_1 \dotsm p_r) \prod_{p < p_r} (1-g(p)). 
\]
As in (i), only $r \geq u\theta-\beta$ contribute and $p_r \geq z_r$. Hence, writing $m = p_1 \dotsm p_r$ and using again $1 \leq 2^{A(\omega(m)-r)}$, we obtain
\[
V_r(z) \leq \prod_{p < z} \left(1-g(p)\right) \prod_{z_r \leq p \leq z} (1+|g(p)|) \sum_{\substack{m \\ p \mid m \implies z_r \leq p < z}} 2^{A(\omega(m)-r)} |\mu(m)| g(m).
\]
Since $|g(p)| \leq 2/p$, we get, using~\eqref{eq:f(n)/n_average} and~\eqref{eq:Mertens},
\[
\begin{split}
V_r(z) &\leq 2^{-Ar} \prod_{p < z} \left(1-g(p)\right) \prod_{z_r \leq p \leq z} \left(1+\frac{2}{p}\right) \left(1+\frac{2^{A+1}}{p}\right) \\
&\ll 2^{-Ar} \left(\frac{\log z}{\log z_r}\right)^{2^{A+1} + 2} \prod_{p < z} \left(1-g(p)\right) \\
&\ll 2^{-Ar} \left(\frac{\beta}{\beta-1}\right)^{(2^{A+1} + 2)r} \prod_{p < z} \left(1-g(p)\right).
\end{split}
\]
Now, once $\beta$ is large enough in terms of $A$,
\[
\sum_{r \geq u\theta - \beta} 2^{-Ar} \left(\frac{\beta}{\beta-1}\right)^{(2^{A+1} + 2)r} \ll_{\beta, A}  2^{-0.9 A u\theta} \leq e^{-Au\theta/2}
\]
and the claim follows.
\end{proof}

\begin{lemma}
\label{le:SieveComb}
Let the assumptions and $\lambda_d$ be as in Lemma~\ref{le:Sieve} and let $v, q \in \mathbb{N}$. Then
\begin{equation}
\label{eq:SieveCombClaim}
\begin{split}
\sum_{\substack{d_1, d_2 | P(z) \\ (d_2,d_1 v)=1 \\ (d_1 d_2, q) = 1}} \frac{\lambda_{d_1} \lambda_{d_2}}{d_1 d_2} = \left(1+O_{\beta, A} \left(e^{-Au\theta/2}\right)\right) \prod_{\substack{p < z \\ p \nmid q}} \left(1-\frac{1}{p}\right) \prod_{\substack{p < z \\ p \nmid vq}} \left(1-\frac{1}{\varphi(p)}\right).
\end{split}
\end{equation}
\end{lemma}
\begin{proof}
We have to be somewhat careful here due to the condition $(d_1, d_2) = 1$ (see also~\cite[Section 5.9]{Opera}).

By~\eqref{eq:lambdadgd} with $g(d) = \frac{\mathbf{1}_{(d, d_2 q) = 1}}{d}$, we can write
\begin{equation}
\label{eq:lamd/d}
\begin{split}
\sum_{\substack{d_1 \mid P(z)  \\ (d_1, d_2q ) = 1}} \frac{\lambda_{d_1}}{d_1} = \prod_{\substack{p < z \\ p \nmid d_2q}} \left(1-\frac{1}{p}\right) - \sum_{\substack{r \geq u\theta-\beta \\ r \text{ odd}}} \sum_{\substack{p_r < p_{r-1} < \dotsc < p_1 < z \\ p_1 p_2 \dotsm p_r p_r^{\beta} \geq D \\ p_1 \dotsm p_h p_h^{\beta} < D \text{ for all odd $h < r$} \\ p_j \nmid d_2q}} \frac{1}{p_1 \dotsm p_r} \prod_{\substack{p < p_r \\ p \nmid d_2q}} \left(1-\frac{1}{p}\right).
\end{split}
\end{equation}
Using Lemma~\ref{le:Sieve}(ii) we see that the first term on the right hand side contributes to the left hand side of~\eqref{eq:SieveCombClaim}
\[
\begin{split}
\sum_{\substack{d_2 \mid P(z)  \\ (d_2, vq) = 1}} \frac{\lambda_{d_2}}{d_2} \prod_{\substack{p < z \\ p \nmid d_2q}} \left(1-\frac{1}{p}\right) &= \prod_{\substack{p < z \\ p \nmid q}} \left(1-\frac{1}{p}\right) \sum_{\substack{d_2 \mid P(z)  \\ (d_2, vq) = 1}} \frac{\lambda_{d_2}}{\varphi(d_2)}\\
& = \prod_{\substack{p < z \\ p \nmid q}} \left(1-\frac{1}{p}\right) (1+O_{\beta, A}(e^{-Au \theta/2})) \prod_{\substack{p < z \\ p \nmid vq}} \left(1-\frac{1}{\varphi(p)}\right). 
\end{split}
\]

Let us now consider the contribution of the $r$-sum in~\eqref{eq:lamd/d} to the left hand side of~\eqref{eq:SieveCombClaim}.
Writing $m = p_1 \dotsm p_{r-1}$ so that $2^{A(\omega(m)-r)} \geq 1/2^A$ and recalling from the proof of Lemma~\ref{le:Sieve} that $p_r \geq z_r$, this contribution is
\[
\begin{split}
& \ll_A \sum_{\substack{r \geq u\theta-\beta}} 2^{-Ar} \sum_{\substack{z_r \leq p_r < z \\ p_r \nmid q}} \frac{1}{p_r}  \sum_{\substack{m \leq D/p_r \\ (m, q) = 1 \\ p \mid m \implies p_r < p < z}}\frac{2^{A\omega(m)}}{m} \prod_{\substack{p < p_r \\ p \nmid q}} \left(1-\frac{1}{p}\right) \\
& \qquad \qquad \qquad \cdot \Biggl|\sum_{\substack{d_2 \mid P(z)  \\ (d_2, vq mp_r) = 1}} \frac{\lambda_{d_2}}{d_2} \prod_{\substack{p < p_r \\ p \mid d_2}} \left(1-\frac{1}{p}\right)^{-1} \Biggr|.
\end{split}
\]
Applying Lemma~\ref{le:Sieve}(ii) and recombining the variables $m$ and $p_r \geq z_r$, and applying then~\eqref{eq:f(n)/n_average} and~\eqref{eq:Mertens}, this is
\[
\begin{split}
&\ll \sum_{\substack{r \geq u\theta-\beta}} 2^{-Ar} \sum_{\substack{m \leq D \\ (m, q) = 1 \\ p \mid m \implies z_r \leq p < z}} \frac{2^{A \omega(m)}}{\varphi(m)} \prod_{\substack{p < z_r \\ p \nmid q}} \left(1-\frac{1}{p}\right) \prod_{\substack{p < z \\ p \nmid vq}} \left(1-\frac{1}{\varphi(p)}\right) \\
&\ll \prod_{\substack{p < z \\ p \nmid vq}} \left(1-\frac{1}{\varphi(p)}\right) \prod_{\substack{p < z \\ p \nmid q}} \left(1-\frac{1}{p}\right)  \sum_{\substack{r \geq u\theta-\beta}} 2^{-Ar} \prod_{\substack{z_r \leq p < z \\ p \nmid q}} \left(1+\frac{2^A+1}{p}\right) \\
&\ll \prod_{\substack{p < z \\ p \nmid vq}} \left(1-\frac{1}{\varphi(p)}\right) \prod_{\substack{p < z \\ p \nmid q}} \left(1-\frac{1}{p}\right) \sum_{\substack{r \geq u\theta-\beta}} 2^{-Ar} \left(\frac{\beta}{\beta-1}\right)^{(2^A+1)r} \\
&\ll_{\beta, A} e^{-Au\theta/2}  \prod_{\substack{p < z \\ p \nmid vq}} \left(1-\frac{1}{\varphi(p)}\right) \prod_{\substack{p < z \\ p \nmid q}} \left(1-\frac{1}{p}\right)
\end{split}
\]
once $\beta$ is large enough.
\end{proof}

\subsection{Poisson summation}
Let us state the version of the Poisson summation formula we shall use. 
\begin{lemma}
\label{le:Poisson}
Let $b \in \mathbb{Z}$, $d, q \in \mathbb{N}$ with $(d, bq) = 1$, let $f \colon \mathbb{R} \to \mathbb{C}$ be such that $f, \widehat{f} \in L^1(\mathbb{R})$ and have bounded variation. Let $g \colon \mathbb{R} \to \mathbb{C}$ be $q$-periodic. Then
\[
\begin{split}
\sum_{m \equiv b \Mod{d}} f(m) g(m) &= \frac{1}{dq}\sum_h \widehat{f}\left(\frac{h}{dq}\right) \sum_{\substack{m \Mod{dq} \\ m \equiv b \Mod{d}}} g(m) e\left(\frac{hm}{dq}\right)
\end{split}
\]
\end{lemma}
\begin{proof}
Write
\[
\sum_{m \equiv b \Mod{d}} f(m) g(m) = \sum_{a \Mod{q}} g(a) \sum_{\substack{m \equiv b \Mod{d} \\ m \equiv a \Mod{q}}} f(m) = \sum_{a \Mod{q}} g(a) \sum_{k} f(c_a+dq k),
\]
where $c_a$ is such that $c_a \equiv b \Mod{d}$ and $c_a \equiv a \Mod{q}$. By Poisson summation (see e.g.~\cite[formula (4.24)]{IwKo}, the above equals
\[
\begin{split}
\sum_{a \Mod{q}} g(a) \frac{1}{dq} \sum_h \widehat{f}\left(\frac{h}{dq}\right) e\left(\frac{hc_a}{dq}\right) &= \frac{1}{dq} \sum_h \widehat{f}\left(\frac{h}{dq}\right) \sum_{a \Mod{q}} g(a) e\left(\frac{hc_a}{dq}\right)\\
&= \frac{1}{dq} \sum_h \widehat{f}\left(\frac{h}{dq}\right) \sum_{\substack{c \Mod{dq} \\ c \equiv b \Mod{d}}} g(c) e\left(\frac{hc}{dq}\right).
\end{split}
\]
\end{proof}

\subsection{Character sums}
\label{ssec:char}
\begin{proof}[Proof of Lemma~\ref{le:Chib2}]
Since $\chi$ is a primitive quadratic character of modulus $q = 2^rq'$ with $2 \nmid q'$, we have that $r \in \{0, 2, 3\}$ and $q'$ is square-free (see e.g.~\cite[Section 3.3]{IwKo}). Hence, by the Chinese reminder theorem it suffices to consider the prime case $q = p > 2$ and the case $q = 2^r$ with $r \in \{2,3\}$. 

\textbf{Proof of~\eqref{eq:chi00}. }
For $p > 2$,
\[
\sum_{m \Mod{p}} \chi_0(m)\chi_0(\pm m+h) = 
\begin{cases}
p-1 = \varphi(p) & \text{if $p \mid h$;} \\
p-2 & \text{if $p \nmid h$.}
\end{cases}
\]
and, for $r \in \{2, 3\}$ and even $h$,
\[
\sum_{m \Mod{2^r}} \chi_0(m)\chi_0(\pm m+h) = \sum_{\substack{m \pmod{2^r} \\ (m, 2) = 1}} 1 = 
2^{r-1} = \varphi(2^r).
\]
Combining these,~\eqref{eq:chi00} follows.

\textbf{Proof of~\eqref{eq:chip0}. }
For $p > 2$,
\[
\sum_{m \Mod{p}} \chi(m)\chi_0(\pm m+h) = \sum_{m \Mod{p}} \chi(m) - \chi(\mp h) = - \chi(\mp h).
\]
and, for $r \in \{2, 3\}$ and even $h$,
\[
\sum_{m \Mod{2^r}} \chi(m)\chi_0(\pm m+h) = \sum_{m \Mod{2^r}} \chi(m) = 0 = -\chi(\mp h).
\]
Combining these,~\eqref{eq:chip0} follows.

\textbf{Proof of~\eqref{eq:chipp}. }
For $p > 2$,
\[
\begin{split}
&\sum_{m \Mod{p}} \chi(m(\pm m+h)) = \sum_{\substack{m \Mod{p} \\ (m, p) = 1}} \chi(m(\pm m+h)) = \chi(\pm 1) \sum_{\substack{m \Mod{p} \\ (m, p) = 1}} \chi(m)^2 \chi(1 \pm h\overline{m}) \\
& =\chi(\pm 1) \sum_{\substack{m \Mod{p} \\ (m, p) = 1}} \chi(1 \pm hm) = \chi(\pm 1) \left(\sum_{\substack{m \Mod{p}}} \chi(1\pm hm) - \chi(1)\right) = \chi(\pm 1) \begin{cases}
p-1 & \text{if $p \mid h$;} \\
-1 & \text{if $p \nmid h$.}
\end{cases}
\end{split}
\]
Similarly, for $r = 2, 3$ and even $h$,
\[
\begin{split}
&\sum_{m \Mod{2^r}} \chi(m(\pm m+h)) = \sum_{\substack{m \Mod{2^r} \\ (m, 2) = 1}} \chi(m(\pm m+h)) = \chi(\pm 1) \sum_{\substack{m \Mod{2^r} \\ (m, 2) = 1}} \chi(m)^2 \chi(1\pm h\overline{m}) \\
& = \chi(\pm 1) \sum_{\substack{m \Mod{2^r} \\ (m, 2) = 1}} \chi(1 \pm hm) = \chi(\pm 1) \left(\sum_{\substack{m \Mod{2^r}}} \chi(1 \pm hm) - \frac{1}{2}\sum_{\substack{m \Mod{2^r}}} \chi(1\pm 2hm)\right)
\end{split}
\]
It is easy to see that, for $r = 2, 3$ and even $h$, the expression in the parentheses equals
\[
\begin{split}
2^r\mathbf{1}_{2^r \mid h} - 2^{r-1} \mathbf{1}_{2^{r-1} \mid h} = 2^{r-1} \mathbf{1}_{2^{r-1} \mid h} (-1)^{\frac{h}{2^{r-1}}} = \varphi(2^r) \mathbf{1}_{\varphi(2^r) \mid h} (-1)^{\frac{h}{\varphi(2^r)}},
\end{split}
\]
where the last formulation is such that it is $1$ for $r = 0$ when $h$ is even. Combining these,~\eqref{eq:chip0} follows.
\end{proof}

\subsection{Kloosterman sums}
For integers $a,b,c$ with $c\geq 1$ we denote the Kloosterman sum by
\[
S(a,b;c) := \sum_{\substack{n \,(c) \\ (n,c)=1}} e_c(an+b\bar{n}).
\]
We shall use the Weil bound (see e.g.~\cite[Corollary 11.12]{IwKo}).
\begin{lemma} \label{kloostermanlemma} Let $a, b \in \mathbb{Z}$ and $c \in \mathbb{N}$. Then, for any $\varepsilon > 0$
\[
S(a,b;c) \ll_{\eps} c^{1/2+\eps}(a,b,c)^{1/2}.
\]
\end{lemma}

\begin{remark}
We could alternatively use a more elementary bound $S(a,b;c) \ll_{\eps} c^{3/4+\varepsilon} (a,b,c)^{1/4}$ due to Kloosterman (see \cite[Lemma 4]{kloosterman} with $\Lambda=1$), and this would only somewhat increase the exponent of $q$ in the conditions of the type $X \geq q^{10}$ and $h \geq q^{10}$ in our theorems and corollaries.
\end{remark}
We also need the following rough bounds.
\begin{lemma}\label{le:gcdsum}
Let $L \geq 1.$ For any integer $q \neq 0$ we have
\[
\sum_{1 \leq \ell \leq L} (\ell, q) \leq \tau(q) L
\]
and
\[
\sum_{1 \leq \ell \leq L} (\ell, q) \frac{\ell}{\varphi(\ell)} \ll \tau(q)^2 L.
\]
\end{lemma}
\begin{proof}
Writing $d = (\ell, q)$ and $\ell = kd$, we have
\[
\sum_{1 \leq \ell \leq L} (\ell, q) \leq \sum_{d | q} d \sum_{1 \leq k \leq \frac{L}{d}} 1 \leq L \sum_{d | q} 1 = \tau(q)L.
\]
Similarly, using also $\varphi(dk) \geq \varphi(d)\varphi(k)$ and $\sum_{k \leq K} \frac{k}{\varphi(k)} \ll K$ we get 
\[
\sum_{1 \leq \ell \leq L} (\ell, q) \frac{\ell}{\varphi(\ell)} \leq \sum_{d| q} \frac{d^2}{\varphi(d)} \sum_{1\leq k \leq L/d} \frac{k}{\varphi(k)} \ll \sum_{d| q} \frac{Ld}{\varphi(d)} = L \prod_{p \mid q} \left(1 + \frac{p}{\varphi(p)}\right) \leq \tau(q)^2 L.
\]
\end{proof}
From Lemma~\ref{kloostermanlemma} we obtain the following bound for incomplete Kloosterman sums via Poisson summation.
\begin{lemma}\label{le:kloosterman}
Let $\delta, \varepsilon > 0$ and $N \geq 1$. Let $F:\mathbb{R} \to \mathbb{C}$ be a bounded smooth compactly supported function and suppose that for some $\delta\in(0,1)$ we have
$ |F''| \ll \delta^{-2}.$ Then, for any integers $\alpha,d,e,q,k$ with $d,q\geq 1$ and $(eq,d)=1$, we have
\[
\sum_{\substack{n \equiv \alpha \, (q) \\ (n,d)=1}} F\bigg( \frac{n}{N}\bigg) e_d(k \overline{en}) \ll_{\eps} \delta^{-1}d^{1/2+\eps} + \frac{N(d,k)}{dq}.
\]
\end{lemma}
\begin{proof}
Applying Poisson summation (Lemma \ref{le:Poisson} with $d$-periodic $g(n) = \mathbf{1}_{(n, d)=1} e_d(k\overline{en})$), we obtain
\begin{equation*}
\sum_{\substack{n \equiv \alpha \, (q) \\ (n,d)=1}} F\bigg( \frac{n}{N}\bigg) e_d(k \overline{en}) = \frac{N}{d q} \sum _{h} \widehat{F}\bigg(\frac{hN}{d q} \bigg) \sum_{\substack{n \, (d q) \\ n \equiv \alpha \, (q) \\ (n,d)=1}}  e_{d}(k \overline{en}) e_{dq}(hn).
\end{equation*} 
Since $(d,q)=1$, by the Chinese remainder theorem we can write $n=\alpha d \bar{d} + \beta q \bar{q}$, where $\bar{d}$ is inverse $\pmod{q}$ and $\bar{q}$ is inverse $\pmod{d}$ to get 
\begin{equation}
\label{eq:KloosAftPoi}
\begin{split}
\sum_{\substack{n \equiv \alpha \, (q)\\ (n,d)=1}} F\bigg( \frac{n}{N}\bigg) e_d(k \overline{en}) &=\frac{N}{d q} \sum _{h} \widehat{F}\bigg(\frac{hN}{d q} \bigg) e_q(h\alpha \bar{d}) \sum_{\substack{\beta\,(d) \\ (\beta,d)=1}} e_d(h\bar{q} \beta + k \overline{e \beta})\\
&= \frac{N}{d q} \sum _{h} \widehat{F}\bigg(\frac{hN}{d q} \bigg) e_q(h\alpha \bar{d}) S(h\bar{q},k\bar{e};d).
\end{split}
\end{equation}

For $h=0$ we have a Ramanujan sum (see e.g.~\cite[(3.1) and (3.5)]{IwKo})
\[
S(0,k\bar{e};d) =  \sum_{\substack{n \,(d) \\ (n,d)=1}} e_{d}(k\overline{en}) \ll (d,k).
\]
The contribution from this to~\eqref{eq:KloosAftPoi} is
\[
\ll \frac{N(d,k)}{d q}| \widehat{F}(0)| \ll  \frac{N (d,k)}{d q}.
\]

For $h \neq 0$ we get by integration by parts
\[
\widehat{F}\bigg(\frac{hN}{d q} \bigg)  \ll  \min\bigg\{1,\bigg|\delta\frac{hN}{dq}\bigg|^{-2}\bigg\}.
\]
Hence, by Lemmas \ref{kloostermanlemma} and \ref{le:gcdsum} we get
\[
\begin{split}
\frac{N}{d q} \sum _{h\neq 0} \widehat{F}\bigg(\frac{hN}{d q} \bigg) e_q(h\alpha \bar{d}) S(h\bar{q},k\bar{e};d) &\ll_\eps  d^{1/2+\eps/2} \frac{N}{d q} \sum _{h\neq 0} (h,d) \min\bigg\{1,\bigg|\delta\frac{h N}{dq}\bigg|^{-2}\bigg\} \\
& \ll_\eps \delta^{-1} d^{1/2+\eps}.
\end{split}
\]
\end{proof}

\section{Proof of Lemma~\ref{le:exzero}}
\label{se:ExcCharCons}
The following lemma is essentially \cite[Proposition 3.5]{TT}, but we provide the proof also here for completeness. 
\begin{lemma} \label{le:exzeroTT}
Let $\chi$ be a primitive quadratic character modulo $q \geq 2$. Assume that $L(s,\chi)$ has a real zero $\beta_0$ such that
\[
\beta_0 = 1-\frac{1}{\eta \log q}
\]
for some $\eta \geq 10$. Let $\delta > 0$. Then, for any $Y > q^{1/2+\delta}$, one has
\[
\sum_{q^{1/2+\delta} < p \leq Y} \frac{\lambda(p)}{p} \ll_\delta \frac{\log Y}{\eta \log q},
\]
and, for any $k \geq 2$, one has
\[
\sum_{q^{(1/2+\delta)/k} < p \leq q^{(1/2+\delta)/(k-1)} } \frac{\lambda(p)}{p} \ll_\delta \frac{k}{\eta^{1/k}}.
\]

\end{lemma}
\begin{proof}
We may assume that $\eta$ is large since otherwise the claims are trivial.
By \cite[Excercise 3(g) in Section 11.2.1]{MVBook} we have for any $y \geq q^{1/2+\delta}$
\[
\sum_{m \leq y} \frac{\lambda(m)}{m} = L(1,\chi)(\log y + \gamma) + L'(1,\chi) + O_\delta(q^{-\delta/3}),
\]
where $\gamma$ is the Euler-Mascheroni constant (see~\cite[formula (1.27)]{MVBook}). Using Siegel's bound  $L(1,\chi) \gg_\delta q^{-\delta/4}$ (see e.g.~\cite[Theorem 11.14]{MVBook}) we get 
\begin{equation} \label{eq:mvasymp2}
\sum_{m \leq y} \frac{\lambda(m)}{m} = L(1,\chi)\left(\log y  + \frac{L'(1,\chi)}{L(1,\chi)} + O_\delta(1) \right).
\end{equation}
By \cite[Theorem 11.4]{MVBook} we have
\[
 \frac{L'(1,\chi)}{L(1,\chi)} = \eta \log q + O(\log q) \asymp \eta \log q
\]
as we assumed that $\eta$ is large.
Plugging this into (\ref{eq:mvasymp2}) with $y=q^{1/2+\delta}$ we obtain
\begin{equation}
\label{eq:etalowerbound}
\sum_{m \leq q^{1/2+\delta}} \frac{\lambda(m)}{m} \gg_\delta \eta L(1,\chi) \log q.
\end{equation}
Also, subtracting (\ref{eq:mvasymp2}) for $y=q^{1/2+\delta}$ and $y=Yq^{1/2+\delta}$ gives
\begin{equation} \label{eq:mvsubtracted}
\sum_{q^{1/2+\delta } < m \leq Y q^{1/2+\delta}} \frac{\lambda(m)}{m} = L(1,\chi)(\log Y + O_\delta(1)).
\end{equation}

On the other hand, by non-negativity and multiplicativity of $\lambda(n)$ we get
\[
\begin{split}
\sum_{q^{1/2+\delta} < m \leq Y q^{1/2+\delta}} \frac{\lambda(m)}{m} &\geq \sum_{\substack{q^{1/2+\delta} < m \leq Y q^{1/2+\delta} \\ m=pn, \, n \leq q^{1/2+\delta} < p}} \frac{\lambda(m)}{m} \geq \sum_{\substack{n \leq q^{1/2+\delta} }} \frac{\lambda(n)}{n} \sum_{\substack{q^{1/2+\delta} < p \leq Y }} \frac{\lambda(p)}{p}.
\end{split}
\]
The first bound in the lemma now follows by (\ref{eq:mvsubtracted}) and (\ref{eq:etalowerbound}).

Similarly, for the second bound we write $m=np_1\cdots p_k$ to get
\begin{equation}
\label{eq:mp1pkm}
\begin{split}
\sum_{q^{1/2+\delta} < m \leq q^3} \frac{\lambda(m)}{m} &\geq \sum_{\substack{q^{1/2+\delta} < m \leq q^3 \\ m=p_1\cdots p_k n, \, n \leq q^{1/2+\delta} \\ q^{(1/2+\delta)/k} < p_1, \dotsc, p_k \leq q^{(1/2+\delta)/(k-1)} \\ p_1,\dots,p_k \nmid n, p_j \text{ distinct}}} \frac{\lambda(m)}{m} \\
\end{split}
\end{equation}
Now, for $m \leq q^3$ we have
\[
\sum_{\substack{p_1\cdots p_k | m \\ q^{(1/2+\delta)/k} < p_1, \dotsc, p_k \leq q^{(1/2+\delta)/(k-1)}}} 1 \leq k! \binom{6k}{k} \leq k! (1+1)^{6k}=k! 2^{6k}.
\] 
Using this and~\eqref{eq:mp1pkm} we see that
\[
\sum_{q^{1/2+\delta} < m \leq q^3} \frac{\lambda(m)}{m}  \gg \frac{1}{k! 2^{6k}} \sum_{\substack{n \leq q^{1/2+\delta} }} \frac{\lambda(n)}{n} \sum_{\substack{q^{(1/2+\delta)/k} < p_1, \dots, p_k \leq q^{(1/2+\delta)/(k-1)} \\ p_1,\dots,p_k \nmid n  \\ p_j \, \text{distinct}}} \frac{\lambda(p_1)\cdots ´\lambda(p_k)}{p_1\cdots p_k}.
\]

Once we have fixed $p_1,\dots,p_{j-1}$, to make sure that $p_{j}$ is distinct from $p_1,\dots,p_{j-1}$ we have to exclude $j-1 \leq k$ primes. To ensure that also $p_j \nmid n$ we have to further remove at most $k$ primes. Let  $\mathcal{P}$ be the set of $2k$  smallest primes  $>q^{(1/2+\delta)/k}$ such that $\lambda(p) =2.$ Then by positivity we get
\[
\sum_{q^{1/2+\delta} < m \leq q^3} \frac{\lambda(m)}{m}  \gg \frac{1}{k! 2^{6k}} \sum_{\substack{n \leq q^{1/2+\delta} }} \frac{\lambda(n)}{n}\bigg(\sum_{\substack{q^{(1/2+\delta)/k} < p \leq q^{(1/2+\delta)/(k-1)} \\ p \notin \mathcal{P}}} \frac{\lambda(p)}{p}\bigg)^k.
\]
Using (\ref{eq:mvsubtracted}) and (\ref{eq:etalowerbound}) we get
\[
\bigg(\sum_{\substack{q^{(1/2+\delta)/k} < p \leq q^{(1/2+\delta)/(k-1)} \\ p \notin \mathcal{P}}} \frac{\lambda(p)}{p}\bigg)^k \ll_\delta \frac{2^{6k}k!}{\eta},
\]
which gives us
\[
\sum_{\substack{q^{(1/2+\delta)/k} < p \leq q^{(1/2+\delta)/(k-1)} \\ p \notin \mathcal{P}}} \frac{\lambda(p)}{p} \ll \frac{k}{\eta^{1/k}}.
\]
The primes $p \in \mathcal{P}$ can be inserted back in with an error term $O(4k q^{-(1/2+\delta)/k})$. By Siegel's bound~\eqref{eq:Siegeleta} this error is $O(k\eta^{-1/k})$, and the second claim follows.
\end{proof}

\begin{proof}[Proof of Lemma~\ref{le:exzero}]
First note that we can assume that 
\begin{equation}
\label{eq:vnuvbound}
\frac{1}{v^{2} \eta^{v/2}} + \frac{v}{\eta} \frac{\log Y}{\log z} \leq 1
\end{equation}
since otherwise the claim follows trivially from $\lambda(m) \leq \tau(m),$~\eqref{eq:f(n)/n_average}, and ~\eqref{eq:Mertens}.

We write 
\begin{equation}
\label{eq:lammmusplit}
\begin{split}
\sum_{\substack{z \leq m \leq Y \\ (m,P(z))=1}}  \frac{\lambda(m)}{m} = &\sum_{\substack{z \leq m \leq Y \\ (m,P(z))=1}}  \frac{|\mu(m)|\lambda(m)}{m} +\sum_{\substack{z \leq m \leq Y \\ (m,P(z))=1}}  \frac{(1-|\mu(m)|)\lambda(m)}{m}.
\end{split}
\end{equation}
For the second sum we get by (\ref{eq:f(n)/n_average}) and  (\ref{eq:Mertens})
\[
\begin{split}
\sum_{\substack{z \leq m \leq Y \\ (m,P(z))=1}}  \frac{(1-|\mu(m)|)\lambda(m)}{m} &\ll  \sum_{p \geq z} \frac{1}{p^2} \sum_{\substack{m \leq Y/p^2 \\ (m,P(z))=1}}  \frac{\tau(m)}{m} \ll \frac{1}{z} \left(\frac{\log Y}{\log z}\right)^2.
\end{split}
\]

In the first sum on the right hand side of~\eqref{eq:lammmusplit} we have
\begin{equation}
\label{eq:mumlamm/m}
\begin{split}
\sum_{\substack{z \leq m \leq Y \\ (m,P(z))=1}}  \frac{|\mu(m)|\lambda(m)}{m} \leq \prod_{z \leq p \leq Y} \left(1+ \frac{\lambda(p)}{p}\right) - 1 &\leq \prod_{z \leq p \leq Y} \exp\left(\frac{\lambda(p)}{p}\right) - 1 \\
&= \exp\left(\sum_{z \leq p \leq Y} \frac{\lambda(p)}{p}\right) - 1.
\end{split}
\end{equation}
Let $\delta > 0$ be small and let 
\[
K := \left\lceil \frac{1/2+\delta}{v} \right\rceil \geq \max\left\{1, \frac{1/2+\delta}{v}\right\},
\]
so that $z = q^v \geq q^{\frac{1/2+\delta}{K}}$. It is easy to see that either $K \leq 2/v$ or $K = 1$. Then by Lemma~\ref{le:exzeroTT}
\[
\begin{split}
\sum_{z \leq p \leq Y} \frac{\lambda(p)}{p} &\leq  \sum_{2 \leq k \leq K} \sum_{q^{(1/2+\delta)/k} < p \leq q^{(1/2+\delta)/(k-1)} } \frac{\lambda(p)}{p} + \sum_{q^{1/2+\delta} < p \leq Y} \frac{\lambda(p)}{p} \\
&\ll_\delta K^2 \eta^{-1/K} + \frac{\log Y}{\eta \log q} \ll \frac{1}{v^2 \eta^{v/2}} + \frac{v}{\eta} \frac{\log Y}{\log z}.
\end{split}
\]
Plugging this in~\eqref{eq:mumlamm/m} and using~\eqref{eq:vnuvbound} we obtain
\[
\sum_{\substack{z \leq m \leq Y \\ (m,P(z))=1}}  \frac{|\mu(m)|\lambda(m)}{m} \leq \exp\left(\sum_{z \leq p \leq Y} \frac{\lambda(p)}{p}\right) - 1 \ll  \frac{1}{v^2 \eta^{v/2}} + \frac{v}{\eta} \frac{\log Y}{\log z}
\]
and the claim follows.
\end{proof}

\section{Proof of Proposition~\ref{prop:typeI2}} \label{se:pfofpropI2}
\label{se:typeI2prop}
Let us study
\[
S^\pm := \sum_{\substack{m_1, m_2, n_1, n_2 \\ \pm m_1 n_1 + h = m_2 n_2 \\ (m_1 m_2 n_1 n_2, P(z)) = 1}} \chi_1(m_1) \chi_2(m_2) \psi_1(n_1) \psi_2(n_2) f\left(\frac{m_1}{M_1}, \frac{m_2}{M_2}, \frac{n_1}{N_1}, \frac{n_2}{N_2}\right).
\]
First we wish to make the variable $n_2$ implicit, so we write the above as
\[
\begin{split}
&\sum_{\substack{m_1, m_2, n_1, n_2 \\ \pm m_1 n_1 + h = m_2 n_2 \\ (m_1 m_2 n_1 n_2, P(z)) = 1}} \chi_1(m_1) \chi_2(m_2) \psi_1(n_1) \psi_2\bigg(\frac{\pm m_1n_1+h}{m_2}\bigg) f\left(\frac{m_1}{M_1}, \frac{m_2}{M_2}, \frac{n_1}{N_1}, \frac{\pm m_1 n_1 + h}{m_2 N_2}\right).
\end{split}
\]

Let us now show that we can replace the conditions $(m_2, P(z)) = 1$ and $(n_2, P(z)) = 1$ by the conditions $(m_2, hP(z)) = 1$ and $(n_2, P_h(z)) = 1$. These two sets of conditions are equivalent unless $(m_2 n_2, h) > 1$. But $(m_2 n_2, h) \mid m_1 n_1$, so in this case there must be a prime $p \geq z$ such that $p \mid m_2 n_2$, $p \mid h$ and $p \mid m_1 n_1$. Using Lemma~\ref{le:Henriot}(i) we see that the error introduced from these changes of summation conditions is
\[
\begin{split}
&\ll \sum_{\substack{p \mid h \\ z \leq p \ll X}} \sum_{\substack{m \ll X/p \\ (\pm m + h/p, P_h(z)) = 1 \\ (m, P(z)) = 1}} \tau(m)\tau(\pm m + h/p) \ll \frac{h}{\varphi(h)} X u^5 \sum_{\substack{p \mid h \\ z \leq p \ll X}} \frac{1}{p \log^2 (X/p)} \\
&\ll \frac{h}{\varphi(h)} \frac{X}{\log^2 X} \frac{u^6}{z}.
\end{split}
\]
Hence we can replace the condition $(m_2 n_2, P(z)) = 1$ by conditions $(m_2, hP(z)) = 1$ and $(n_2, P_h(z)) = 1$.

Next we shall replace the conditions $(m_1 n_1, P(z))=1$ and $(n_2, P_h(z)) = 1$ by sieve weights. By Lemma~\ref{le:Sieve}(i) with $\theta = 1/1000$ we can write
\[
\begin{split}
\mathbf{1}_{(n_2, P_h(z)) = 1} &= \sum_{\substack{d_2 \mid (n_2, P(z)) \\ (d_2, h) = 1}} \lambda_{d_2} + O\left(\sum_{r_2 \geq u/1000 - \beta} 2^{-Ar_2} \mathbf{1}_{(n_2, P_h(z_{r_2})) = 1} \tau(n_2)^{A+1}\right) \\
\mathbf{1}_{(n_1, P(z)) = 1} &= \sum_{\substack{d_1 \mid (n_1, P(z))}} \lambda_{d_1} + O\left(\sum_{r_1 \geq u/1000 - \beta} 2^{-Ar_1} \mathbf{1}_{(n_1, P(z_{r_1})) = 1} \tau(n_1)^{A+1}\right) \\
\mathbf{1}_{(m_1, P(z)) = 1} &= \sum_{\substack{e \mid (m_1, P(z))}} \lambda_{e} + O\left(\sum_{r_3 \geq u/1000 - \beta} 2^{-Ar_3} \mathbf{1}_{(m_1, P(z_{r_3})) = 1} \tau(m_1)^{A+1}\right).
\end{split}
\]
Using this, and noticing that e.g. $\mathbf{1}_{(m_1, P(z)) = 1} \leq 2^{-A \cdot 0} \mathbf{1}_{(m_1, P(z_{0})) = 1} \tau(m_1)^{A+1}$, we obtain
\[
\begin{split}
S^{\pm} &= \sum_{\substack{d_1 d_2 | P(z) \\ (d_2,h)=1}} \lambda_{d_1} \lambda_{d_2} \sum_{\substack{e |P(z)}}  \lambda_e \sum_{\substack{m_1,m_2 \\( m_2,h P(z))=1}} \chi_1(em_1) \chi_2(m_2) \psi_2(m_2) W(d_1, d_2, em_1, m_2)\\
& + O\left(\frac{h}{\varphi(h)} \frac{X}{\log^2 X} \frac{u^6}{z}\right) + O\Biggl(\sum_{\substack{r_1, r_2, r_3 \geq 0 \\ \max r_j \geq \frac{u}{1000}-\beta}} 2^{-A(r_1 +r_2 +r_3)} \\
& \quad \cdot \sum_{\substack{m_1, m_2, n_1, n_2 \\ \pm m_1 n_1 +h = m_2 n_2 \\  (n_1, P(z_{r_1})) = (n_2, P_h(z_{r_2}))= 1 \\ (m_1, P(z_{r_3})) = (m_2, h P(z)) = 1}} (\tau(m_1) \tau(n_1) \tau(n_2))^{A+1}  f\left(\frac{m_1}{M_1}, \frac{m_2}{M_2}, \frac{n_1}{N_1}, \frac{\pm m_1 n_1 + h}{m_2 N_2}\right) \Biggr)
\end{split}
\]
with
\[
\begin{split}
W(d_1,d_2,em_1,m_2) &:=
\sum_{\substack{n_1 \\ \pm e m_1 d_1 n_1 \equiv -h  \Mod{d_2m_2}}} \psi_1(d_1 n_1) \psi_2(\pm e m_1 d_1 n_1+h) \\
& \qquad \cdot f\left(\frac{em_1}{M_1}, \frac{m_2}{M_2}, \frac{d_1n_1}{N_1}, \frac{\pm em_1 d_1n_1 + h}{m_2 N_2}\right).
\end{split}
\]
We can concentrate on the main term. Note that since $(d_2 m_2, h) = 1$, we can add the condition $(e m_1 d_1, d_2m_2) = 1$. Due to the support of $\chi$ and $\psi$ (noting that $d_2 m_2 \mid \pm e m_1 d_1 n_1 + h$) we can also add conditions $(d_1 d_2 e m_1 m_2, q) = 1$. Hence we need to study
\begin{equation}
\label{eq:SumWconds}
\begin{split}
 \sum_{\substack{d_1 d_2 | P(z) \\ (d_2,d_1 h)=1 \\ (d_1 d_2, q) = 1}} \lambda_{d_1} \lambda_{d_2} \sum_{\substack{e |P(z) \\ (e, d_2 q) = 1}}  \lambda_e \sum_{\substack{m_1,m_2 \\( m_2,h P(z))=1 \\ (e m_1 d_1, d_2 m_2) = 1 \\ (m_1 m_2, q) = 1}} \chi_1(em_1) \chi_2(m_2) \psi_2(m_2) W(d_1, d_2, em_1, m_2).
\end{split}
\end{equation}

Now
\[
\begin{split}
&\int f\left(\frac{em_1}{M_1}, \frac{m_2}{M_2}, \frac{d_1 y}{N_1}, \frac{\pm em_1 d_1 y + h}{m_2 N_2}\right) e(ky) dy \\
& = \frac{1}{em_1 d_1}  \int f\left(\frac{em_1}{M_1}, \frac{m_2}{M_2}, \frac{y}{em_1 N_1}, \frac{\pm y + h}{m_2 N_2}\right) e\left(k\frac{y}{em_1 d_1}\right) dy,
\end{split}
\]
so by Poisson summation (Lemma~\ref{le:Poisson}) we get 
\begin{equation}
\label{eq:WAfterPoisson}
\begin{split}
&W(d_1, d_2, em_1, m_2) = \frac{1}{d_1 d_2 em_1 m_2 q} \sum_{k \in \mathbb{Z}} \int f\left(\frac{em_1}{M_1}, \frac{m_2}{M_2}, \frac{y}{em_1 N_1}, \frac{\pm y+h}{m_2 N_2}\right) \\
& \quad \cdot e\left(k\frac{y}{d_1 d_2 em_1 m_2 q}\right) dy \sum_{\substack{n \, (q d_2 m_2) \\ \pm n \equiv -h \overline{d_1 em_1} \Mod{d_2 m_2}}} \psi_1(d_1 n) \psi_2(\pm em_1 d_1n + h) e_{qd_2 m_2}(kn).
\end{split}
\end{equation}
By the Chinese remainder theorem we can write
\[ 
n = \mp q \bar{q} h \overline{d_1 e m_1} + d_2m_2 \gamma,
\]
so that
\begin{equation}
\label{eq:PoisChinese}
\begin{split}
&\sum_{\substack{n \, (q d_2 m_2) \\ n \equiv -h \overline{d_1 em_1} \Mod{d_2 m_2}}} \psi_1(d_1 n) \psi_2(\pm em_1 d_1n + h) e_{qd_2 m_2}(kn) \\
&= e_{d_2m_2}(\mp \bar{q} hk \overline{d_1 e m_1}) \sum_{\substack{\gamma \,(q) }} \psi_1(d_1d_2m_2 \gamma) \psi_2(\pm \gamma d_1d_2 em_1m_2 + h) e_{q}(k\gamma).
\end{split}
\end{equation}
\subsection{Contribution from $k=0$}
Since in~\eqref{eq:SumWconds} we have $(em_1 d_1 d_2 m_2,q) = 1$, writing $\gamma' = \gamma e m_1 d_1 d_2 m_2$ we get
\[
\begin{split}
\sum_{\substack{\gamma \,(q) }} \psi_1(d_1 d_2m_2 \gamma) \psi_2(\pm \gamma d_1d_2 em_1m_2 + h) &=  \psi_1(em_1) \sum_{\substack{\gamma' \,(q) }} \psi_1(\gamma')\psi_2(\pm \gamma'  + h) \\
& =: \psi_1(em_1) U^{\pm}(h;q),
\end{split} 
\]
say. Recombining the variables $e, m_1$, we see that the contribution from the part of~\eqref{eq:WAfterPoisson} with $k=0$ to~\eqref{eq:SumWconds} is
\begin{equation}
\label{eq:Um1m2sumlamd1d2e}
\begin{split}
&\frac{U^{\pm}(h;q)}{q} \sum_{\substack{m_1,m_2 \\( m_2,hP(z))=1 \\ (m_1, m_2) = 1}} \frac{\chi_1(m_1) \chi_2(m_2) \psi_1(m_1) \psi_2(m_2)}{m_1 m_2}  \Biggl(\sum_{\substack{d_1, d_2 | P(z) \\ (d_2,d_1 h m_1)=1 \\ (d_1 d_2, q) = 1}} \frac{\lambda_{d_1} \lambda_{d_2}}{d_1 d_2} \Biggr) \\
& \qquad \cdot    \int f\left(\frac{m_1}{M_1}, \frac{m_2}{M_2}, \frac{y}{m_1 N_1}, \frac{\pm y + h}{m_2 N_2}\right) dy \sum_{\substack{e |(m_1, P(z))}}  \lambda_e
\end{split}
\end{equation}
(here we were able to drop the condition $(d_1, m_2) = 1$ since $d_1 \mid P(z)$ and $(m_2, P(z)) = 1$).
Notice that by Lemma~\ref{le:SieveComb} with $\theta= 1/1000$ and~\eqref{eq:chi00} we have, for $m_1 \in \mathbb{N}$,
\begin{equation}
\label{eq:Ud1d2}
\begin{split}
&\frac{U^{\pm}(h;q)}{q} \sum_{\substack{d_1, d_2 | P(z) \\ (d_2,d_1 h m_1)=1 \\ (d_1 d_2, q) = 1}} \frac{\lambda_{d_1} \lambda_{d_2}}{d_1 d_2} \\
&=\frac{U^{\pm}(h;q)}{q} \left(1+O_{A} \left(e^{-Au/2000}\right)\right) \prod_{\substack{p < z \\ p \nmid q}} \left(1-\frac{1}{p}\right) \prod_{\substack{p < z \\ p \nmid hqm_1}} \left(1-\frac{1}{\varphi(p)}\right) \\
&\ll \frac{u^2}{\log^2 X} \frac{h}{\varphi(h)}\frac{m_1}{\varphi(m_1)}.
\end{split}
\end{equation}

Using Lemma~\ref{le:Sieve}(i) and~\eqref{eq:Ud1d2} we see that replacing the sum over $e$ in~\eqref{eq:Um1m2sumlamd1d2e} by $\mathbf{1}_{(m_1, P(z)) = 1}$ affects~\eqref{eq:Um1m2sumlamd1d2e} by
\[
\ll \frac{h}{\varphi(h)} \frac{X}{\log^2 X}  u^2 \sum_{r \geq u/1000-\beta} 2^{-Ar} \sum_{\substack{m_1 \ll M_1 ,m_2 \ll M_2 \\(m_1, P(z_r)) = ( m_2,P(z))=1}} \frac{\tau(m_1)^{A+1}}{\varphi(m_1) m_2}.
\]
Using~\eqref{eq:f(n)/n_average} and~\eqref{eq:Mertens}, and taking $\beta$ sufficiently large in terms of $A$, this is, as in our earlier arguments,
\[
\ll_A \frac{h}{\varphi(h)} \frac{X}{\log^2 X} e^{-Au/3000}
\]
Using also the equality in~\eqref{eq:Ud1d2}, handling the error term with~\eqref{eq:f(n)/n_average} and~\eqref{eq:Mertens}, we obtain that~\eqref{eq:Um1m2sumlamd1d2e} equals
\[
\begin{split}
&\frac{U^{\pm}(h;q)}{q} \sum_{\substack{m_1,m_2 \\(m_1 m_2,P(z))=1 \\ (m_1, m_2) = 1}} \frac{\chi_1(m_1) \chi_2(m_2) \psi_1(m_1) \psi_2(m_2)}{m_1 m_2}  \prod_{\substack{p < z \\ p \nmid q}} \left(1-\frac{1}{p}\right) \prod_{\substack{p < z \\ p \nmid hq}} \left(1-\frac{1}{\varphi(p)}\right) \\
& \qquad \cdot    \int f\left(\frac{m_1}{M_1}, \frac{m_2}{M_2}, \frac{y}{m_1 N_1}, \frac{\pm y + h}{m_2 N_2}\right) dy  +  O\left(\frac{h}{\varphi(h)} \frac{X}{\log^2 X} e^{-Au/3000}\right)
\end{split}
\]
Finally we need to remove the condition $(m_1, m_2) = 1$. Since $(m_1 m_2, P(z)) = 1$, using~\eqref{eq:Ud1d2} and then~\eqref{eq:f(n)/n_average} and~\eqref{eq:Mertens}, this introduces an error
\[
\ll \frac{h}{\varphi(h)}\frac{X}{\log^2 X} u^2 \sum_{z \leq p \leq X} \frac{1}{p(p-1)}\sum_{\substack{m_1,m_2 \leq X \\(m_1 m_2,P(z))=1}} \frac{1}{\varphi(m_1) m_2} \ll \frac{h}{\varphi(h)}\frac{X}{\log^2 X} \frac{u^4}{z}.
\]
\subsection{Contribution from $k\neq 0$}
By partial integration we have, for $k \neq 0$,
\[
\begin{split}
&\int f\left(\frac{em_1}{M_1}, \frac{m_2}{M_2}, \frac{y}{em_1 N_1}, \frac{\pm y + h}{m_2 N_2}\right) e\left(k\frac{y}{d_1 d_2 em_1 m_2 q}\right) dy \\
&\ll_j \left(\frac{k}{d_1 d_2 M_1 M_2 q}\right)^{-j} \left(\frac{\delta^{-1}}{M_1 N_1} + \frac{\delta^{-1}}{M_2 N_2}\right)^j \ll_j \left(\frac{\delta^{-1} d_1 d_2 M_1 M_2 q}{k X}\right)^j.
\end{split}
\]
Hence $|k| > X^\eps \frac{ q d_1 d_2 M_1 M_2}{\delta X}$ contribute $O(X^{-100})$ to~\eqref{eq:WAfterPoisson}. The remaining part contributes to~\eqref{eq:SumWconds} by~\eqref{eq:WAfterPoisson} and~\eqref{eq:PoisChinese}
\begin{equation}
\label{eq:kneq0contr}
\begin{split}
&\sum_{\substack{d_1 d_2 | P(z) \\ (d_2,d_1 h)=1 \\ (d_1 d_2, q) = 1}} \lambda_{d_1} \lambda_{d_2} \sum_{\substack{e |P(z) \\ (e, d_2 q) = 1}}  \lambda_e \sum_{\substack{m_1,m_2 \\( m_2,hP(z))=1 \\ (e m_1 d_1, d_2 m_2) = 1 \\ (m_1 m_2, q) = 1}}\chi_1(em_1) \chi_2(m_2) \psi_2(m_2) \\
&\cdot \frac{1}{q d_1 d_2 e m_1 m_2} \sum_{0 < |k| \leq \frac{ q d_1 d_2 M_1 M_2}{\delta X^{1-\varepsilon}}} \int f\left(\frac{em_1}{M_1}, \frac{m_2}{M_2}, \frac{y}{em_1 N_1}, \frac{\pm y + h}{m_2 N_2}\right) e\left(k\frac{y}{d_1 d_2 em_1 m_2 q}\right) dy \\
& \qquad \cdot e_{d_2m_2}(\mp \bar{q} hk \overline{d_1 e m_1}) \sum_{\substack{\gamma \,(q) }} \psi_1(d_1d_2m_2 \gamma) \psi_2(\pm \gamma d_1d_2 em_1m_2 + h) e_{q}(k\gamma) 
\end{split}
\end{equation}
We write
\[
\begin{split}
&\int f\left(\frac{em_1}{M_1}, \frac{m_2}{M_2}, \frac{y}{em_1 N_1}, \frac{\pm y+h}{m_2 N_2}\right) e\left(k\frac{y}{d_1 d_2 em_1 m_2 q}\right) dy \\
&= m_1 \int f\left(\frac{em_1}{M_1}, \frac{m_2}{M_2}, \frac{y}{e N_1}, \frac{\pm y m_1 + h}{m_2 N_2}\right) e\left(k\frac{y}{d_1 d_2 e m_2 q}\right) dy
\end{split}
\]
and rearrange~\eqref{eq:kneq0contr} so that the sum over $m_1$ is innermost. Splitting also the sum over $m_1$ into congruence classes modulo $q$,~\eqref{eq:kneq0contr} is bounded by
\begin{equation}
\label{eq:kneq0contr2}
\begin{split}
&\ll_\eps q^2 X^{\eps} \sum_{\substack{e |P(z) \\ e \leq D}} \sum_{\substack{(m_2,P(z))=1\\ m_2 \ll M_2}}  \sum_{\substack{d_1 d_2 | P(z) \\ d_1,d_2 \leq D \\ (d_2,h)=1}} \mathbf{1}_{(d_1 e q, d_2 m_2) = 1} \\
&\quad \cdot \frac{X}{q d_1 d_2 M_1 M_2} \sum_{0 < |k| \leq \frac{ q d_1 d_2 M_1 M_2}{\delta X^{1-\varepsilon}}} \left|\Upsilon(h,k,e,q, d_1; d_2 m_2)\right|,
\end{split}
\end{equation}
where, for some reduced residue $\alpha \, (q)$ and $y \asymp eN_1$, we have
\[
\Upsilon(h,k, e, q, d_1 ; d_2 m_2) := \sum_{\substack{m_1 \equiv \alpha \, (q) \\ (m_1, d_2 m_2) = 1}} f\left(\frac{em_1}{M_1}, \frac{m_2}{M_2}, \frac{y}{eN_1}, \frac{\pm y m_1 + h}{m_2 N_2}\right)  e_{d_2m_2}(\mp \bar{q} hk \overline{d_1 e m_1})
\]
which is an incomplete Kloosterman sum modulo $d_2m_2$. Applying Lemma \ref{le:kloosterman} we get
\[
\Upsilon(h,k, e, q, d_1 ; d_2 m_2) \ll_\eps \delta^{-1}(d_2 m_2)^{1/2 + \eps} + \frac{M_1 (hk,d_2m_2)}{e d_2m_2q}.
\]
Hence~\eqref{eq:kneq0contr2} is bounded by (using Lemma \ref{le:gcdsum} and that $D = X^{1/1000}$ and $M_2 \ll (M_2 N_2)^{1/2} \ll (\delta^{-1} X)^{1/2}$)
\[
\begin{split}
&\ll_\eps q^2 X^{2\eps} \bigg(\delta^{-2}D^{3+1/2} \sum_{m_2 \ll M_2} m_2^{1/2} \\
&\qquad \qquad + \sum_{e \leq D}\sum_{\substack{d_1, d_2\leq D \\ m_2 \ll M_2}} \frac{X}{q d_1 d_2 M_1 M_2} \sum_{0 < |k| \leq \frac{ q d_1 d_2 M_1 M_2}{\delta X^{1-\varepsilon}}} \frac{M_1 (hk,d_2m_2)}{e d_2m_2q} \bigg) \\
&\ll_\eps   q^2 X^{2\eps} \delta^{-2} D^{7/2} (\delta^{-1} X)^{(1/2) \cdot (3/2)} +  q X^{4\eps} \delta^{-1} D M_1 \ll  \delta^{-3} X^{7/9} q^2,
\end{split}
\]
and the claim follows.

\section{Proof of Lemma \ref{le:chisums}}
\label{se:chisums}
Recalling that $z=X^{1/u}$, by (\ref{eq:Mertens}) it suffices to show that
\[
\begin{split}
&\sum_{\substack{n \leq N \\ (n,P(z))=1}} \frac{\chi(n)\log(y/n)}{n } = \prod_{p < z} \left(1-\frac{1}{p}\right)^{-1}  + O\left(\left( \frac{u^3}{v^2 \eta^{v/2}} + \frac{u^4v}{\eta}\right)\log X \right) \\
&\qquad + O_A\left(\left(\frac{u^3}{z} +e^{-A u/3000}\right)\log X\right) + O_{\varepsilon, C}\left(\exp(- C\log^{3/5-\eps} X)\right).
\end{split}
\]
We first replace $\chi(n)$ by $\mu(n)$ in the sum. We have
\[
\chi(n) = (\lambda\ast \mu)(n) = \mu(n) + \sum_{\substack{n=km \\ m > 1}} \mu(k)\lambda(m),
\]
so that (analogously to~\eqref{eq:lambda'-Lambda})
\begin{equation}
\label{eq:GNdec}
\begin{split}
\sum_{\substack{n \leq N \\ (n,P(z))=1}} \frac{\chi(n)\log (y/n)}{n} =& \sum_{\substack{n \leq N \\ (n,P(z))=1}} \frac{\mu(n)\log (y/n)}{n}  + \sum_{\substack{km \leq N \\ (km,P(z))=1 \\ m \geq z}} \frac{\mu(k)\lambda(m) \log \frac{y}{km}}{km}.
\end{split}
\end{equation}
Using Lemma \ref{le:exzero} the second term gives (by (\ref{eq:f(n)/n_average}) and (\ref{eq:Mertens})) an admissible contribution 
\[
\ll \log X \sum_{\substack{ k \leq N \\ (k,P(z))=1}} \frac{1}{k}  \sum_{\substack{z \leq m \leq N \\ (m,P(z))=1}} \frac{\lambda(m)}{m} \ll u \log X \left(\frac{1}{v^2\eta^{v/2}} + \frac{uv}{\eta} + \frac{1}{z} \right) u^2.
\]
Writing $\lambda_L(n)$ for the Liouville function, the first sum on the right hand side of~\eqref{eq:GNdec} is by~\eqref{eq:f(n)/n_average} and~\eqref{eq:Mertens}
\[
\begin{split}
&\sum_{\substack{n \leq N \\ (n,P(z))=1}} \frac{\lambda_L(n)\log y/ n}{n} + O\left(\log X \sum_{p \geq z} \frac{1}{p^2} \sum_{\substack{n \leq N/p^2 \\ (n,P(z))=1}} \frac{1}{n}\right) \\
&=  \sum_{\substack{n \leq N \\ (n,P(z))=1}} \frac{\lambda_L(n)\log y/ n}{n} + O\left(\frac{u \log X}{z}\right). 
\end{split}
\]

We deal with the condition $(n, P(z)) = 1$ using Lemma~\ref{le:Sieve}(i) with $\theta = 1/1000$. This gives
\[
\begin{split}
\sum_{\substack{n \leq N \\ (n,P(z))=1}} \frac{\lambda_L(n)\log y/ n}{n} &= \sum_{d \mid P(z)} \frac{\lambda_d \lambda_L(d)}{d} \sum_{n \leq N/d} \frac{\lambda_L(n) \log \frac{y}{dn}}{n} \\
 & \quad + O\left(\log X \sum_{r \geq u/1000-\beta} 2^{-Ar} \sum_{n \leq N} \mathbf{1}_{(n, P(z_r)) = 1} \frac{\tau(n)^{A+1}}{n}\right).
\end{split}
\]
By~\eqref{eq:f(n)/n_average} and~\eqref{eq:Mertens} the error term is
\[
\begin{split}
&\ll_A \log X \sum_{r \geq u/1000-\beta} 2^{-Ar} \left(\frac{\log X}{\log z_r}\right)^{2^{A+1}} \ll \log X  \sum_{r \geq u/1000-\beta} 2^{-Ar} u^{2^{A+1}}\left(\frac{\beta}{\beta - 1}\right)^{2^{A+1}r} \\
&\ll_A  e^{-Au/3000} \log X 
\end{split}
\]
when $\beta$ is sufficiently large in terms of $A$.

For $\Re s > 1$, let
\[
F(s):= \sum_{n \in \mathbb{N}} \frac{\lambda_L(n)}{n^s} = \prod_{p \in \mathbb{P}} \left(1+\frac{1}{p^s}\right)^{-1} = \frac{1}{\zeta(s)} \prod_{p \in \mathbb{P}} \frac{1}{1-\frac{1}{p^{2s}}},
\]
so that
\[
\sum_n \frac{\lambda_L(n) \log\frac{y}{dn}}{n^s} = F(s)\log\frac{y}{d} + F'(s).
\]
Note that $F(s)$ has a simple zero at $s=1$ whereas $F'(s)$ is holomorphic but non-zero at $s=1$. 

By Perron's formula (see e.g.~\cite[Corollary 5.3]{MVBook}) we have, for  $T := \exp(2C\log^{3/5-\varepsilon} X)$,
\[
\begin{split}
&\sum_{d \mid P(z)} \frac{\lambda_d \lambda_L(d)}{d} \sum_{n \leq N/d} \frac{\lambda_L(n) \log \frac{y}{dn}}{n} \\
&= \sum_{d \mid P(z)} \frac{\lambda_d \lambda_L(d)}{d} \frac{1}{2 \pi i}\int_{1/\log N-iT}^{1/\log N+iT}  \left(F(s+1)\log\frac{y}{d} + F'(s+1)\right) \frac{(N/d)^s}{s} ds + O\left(\frac{\log^5 X}{T}\right).
\end{split}
\]
We move the integration line to $\Re s = - 1/\log^{2/3+\varepsilon} T$, staying in the zero-free region for $\zeta(s)$ (see e.g.~\cite[Theorem 8.29]{IwKo}). From the residue at $s=0$ we get a main term
\[
\sum_{d \mid P(z)} \frac{\lambda_d \lambda_L(d)}{d} F'(1) = \sum_{d \mid P(z)} \frac{\lambda_d \lambda_L(d)}{d} \prod_{p \in \mathbb{P}} \frac{1}{1-\frac{1}{p^2}}.
\]
By Lemma~\ref{le:Sieve}(ii) this equals
\[
\begin{split}
(1+O_A(e^{-Au/2000})) \prod_{p < z} &\left(1+\frac{1}{p}\right) \prod_{p \in \mathbb{P}} \frac{1}{\left(1+\frac{1}{p}\right)\left(1-\frac{1}{p}\right)} \\
&= \left(1+O_A\left(e^{-Au/2000} + \frac{1}{z}\right)\right)  \prod_{p < z} \left(1-\frac{1}{p}\right)^{-1}.
\end{split}
\]
Let us now consider the remaining integral. For $s = -1/\log^{2/3+\varepsilon} T + it$ with $|t| \leq T$ one has (see e.g.~\cite[Theorem 8.29]{IwKo})
\[
\frac{1}{\zeta(s+1)} \ll \log^{2/3+\varepsilon} T \quad \text{and} \quad \frac{\zeta'(s+1)}{\zeta(s+1)} \ll \log^{2/3+\varepsilon} T.
\]
Hence the remaining integral contributes
\[
\begin{split}
&\ll \sum_{d \leq N^{1/1000}} \frac{1}{d} \left|\int_{-1/\log^{2/3+\varepsilon} T -iT}^{-1/\log^{2/3+\varepsilon} T + iT}  \left(F(s+1)\log\frac{y}{d} - F'(s+1)\right) \frac{(N/d)^s}{s} ds \right| \\
&\ll \log^2 X \cdot \log^3 T \cdot N^{-1/(2\log^{2/3+\varepsilon} T)} \ll_{C, \varepsilon} \exp(-C\log^{3/5-\varepsilon} X),
\end{split}
\]
and the claim follows.

\section{Proof of Theorems~\ref{th:MTzero} and~\ref{th:Goldbach}} 
\label{se:twinPrimeThms}
Theorems~\ref{th:MTzero} and~\ref{th:Goldbach} are trivial unless $h$ is even. Furthermore they follow from Lemma~\ref{le:Henriot}(i) with $w_1 = w_2 = X^{1/4}$ and $m_1 = m_2 = 0$ unless $\eta$ is large. We will assume these as well as $\eta \ll q^\varepsilon$ from~\eqref{eq:Siegeleta}.

As explained in beginning of Section~\ref{se:initial}, it suffices to study, with $g$ as there,
\[
\sum_{n} g\left(\frac{n}{X}\right) \Lambda(n) \Lambda(\pm n+h),
\]
where in case of $+$-sign we have $0 \neq |h| \leq X^{1+\varepsilon/2}$ and in case of $-$-sign we have $X \in [h^{1-\varepsilon/3}, h/4]$, for some small $\varepsilon > 0$. Here we have done a dyadic splitting of $n$, discarding $n \leq X^{1-\varepsilon/2}$. Hence now $q = X^{1/V'}$ for some $V' \in [(1-\varepsilon/2)V, V]$.

Define $z := X^{1/u}$ for some $u$ to be chosen shortly (in~\eqref{eq:uchoice}). Then $z = q^{V'/u}$.  By~\eqref{eq:Lnn+hSplit} and Lemmas~\ref{le:cncor} and~\ref{le:exzero} we have
\[
\begin{split}
&\sum_{n} g\left(\frac{n}{X}\right) \Lambda(n) \Lambda(\pm n+h) = \sum_{\substack{n \\ (n(\pm n+h), q P(z)) = 1}}  g\left(\frac{n}{X}\right) \lambda'(n) \lambda'(\pm n+h) \\
&\qquad + O\left(\frac{h}{\varphi(h)} X \left(\frac{u^8}{V^2 \eta^{V/(3u)}} + \frac{u^6 V}{\eta} + \frac{u^{10}}{q^{V/(2u)}} \right) + (z+\omega(q)) \log^2 X\right).
\end{split}
\]
To balance the error terms with errors of the type $e^{-Au/3000}$ that we will encounter later, we choose
\begin{equation}
\label{eq:uchoice}
u = \min\left\{\frac{\sqrt{V \log \eta}}{10 C}, \log \eta\right\}.
\end{equation}
With this choice, the error terms are acceptable and we can concentrate on the main term.

By the definition of $\lambda'$ (see~\eqref{eq:deflamlam'}), here
\begin{equation}
\label{eq:lam'lam'sum}
\begin{split}
&\sum_{\substack{n \\ (n(\pm n+h), q P(z)) = 1}}  g\left(\frac{n}{X}\right) \lambda'(n) \lambda'(\pm n+h) \\
&= \sum_{\substack{m_1, m_2, n_1, n_2 \\ \pm m_1 n_1 + h = m_2 n_2 \\(m_1 n_1 m_2 n_2, P(z)) = 1}}  g\left(\frac{m_1 n_1}{X}\right) \chi(m_1) \chi(m_2) \chi_0(n_1) \chi_0(n_2) \log n_1 \log n_2.
\end{split}
\end{equation}
Let $X_{\pm} := \pm X + h \asymp \max\{X, h\} \ll X^{1+\varepsilon/2}$, so that $m_2 n_2 \asymp X_{\pm}$.
We make a smooth partition of the variables $m_j$. We let $\Psi\colon \mathbb{R} \to [0, 1]$ be a smooth function for which $\Psi(x) = 0$ for $x \leq 1$ and $\Psi(x) = 1$ for $x \geq 2$. Define then $F \colon \mathbb{R}_+ \to [0, 1]$ by
\[
F(x) = \begin{cases}
\Psi(x) & \text{if $0 < x \leq 2$} \\
1-\Psi(x/2) & \text{if $x > 2$}.
\end{cases}
\]  
The function $F(x)$ is supported on $[1, 4]$ and gives a smooth partition of unity
\[
\sum_{j \in \mathbb{Z}} F\left(\frac{x}{2^j}\right) = 1 \quad \text{for all $x > 0$}.
\]
We write $N_1 = X/M_1$ and $N_2 = X_{\pm}/M_2 \asymp \max\{X, h\}/M_2$. Then when $m_j \in [M_j, 4M_j]$,  in~\eqref{eq:lam'lam'sum}, then $n_j \in [N_j/20, 20N_j]$. Let $h \colon \mathbb{R}_{\geq 0} \to [0, 1]$ be a smooth function supported on $[1/40, 40]$ such that $h(x)=1$ for $x \in [1/20, 20]$, and write
\[
f_{M_1, M_2}(x_1, x_2, y_1, y_2) = g(x_1 y_1) F(x_1) F(x_2)  h(y_1) h(y_2) \frac{\log \left(y_1N_1\right) \log\left(y_2N_2\right)}{\log^2 X}.
\]
Then
\[
\begin{split}
&\sum_{\substack{n \\ (n(\pm n+h), q P(z)) = 1}}  g\left(\frac{n}{X}\right) \lambda'(n) \lambda'(\pm n+h) \\
&=\log^2 X \sum_{\substack{M_1 = 2^{i_1}, M_2 = 2^{i_2} \\ 1/4 \leq M_1 \leq 4X \\ 1/4 \leq M_2 \leq 4X_{\pm}}} \sum_{\substack{m_1, m_2, n_1, n_2 \\ \pm m_1 n_1 + h = m_2 n_2 \\(m_1 n_1 m_2 n_2, P(z)) = 1}} f_{M_1, M_2}\left(\frac{m_1}{M_1}, \frac{m_2}{M_2}, \frac{n_1}{N_1}, \frac{n_2}{N_2}\right) \\
& \hspace{200pt} \cdot \chi(m_1) \chi(m_2) \chi_0(n_1) \chi_0(n_2) \\
&=: \Sigma^{\pm}_{S, S} + \Sigma^{\pm}_{S, L} + \Sigma^{\pm}_{L, S} + \Sigma^{\pm}_{L, L},
\end{split}
\]
say, where $\Sigma^{\pm}_{S, S}$ corresponds to $M_1 \leq X^{1/2}$ and $M_2 \leq X_{\pm}^{1/2}$, $\Sigma^{\pm}_{S, L}$ corresponds to $M_1 \leq X^{1/2}$ and $M_2 > X_{\pm}^{1/2}$, $\Sigma^{\pm}_{L, S}$ corresponds to $M_1 > X^{1/2}$ and $M_2 \leq X_{\pm}^{1/2}$, and $\Sigma^{\pm}_{L, L}$ corresponds to $M_1 \geq X^{1/2}$ and $M_2 \geq X_{\pm}^{1/2}$. 

Let us start with $\Sigma^{\pm}_{S, S}$. Applying Proposition~\ref{prop:typeI2} with $\psi_1 = \psi_2 = \chi_0, \chi_1 = \chi_2 = \chi$, and $\delta = X^{-1/1000}$, handling the error term with~\eqref{eq:ConsHenriot}, we see that, for any $A \geq 1$, 
\[
\begin{split}
&\Sigma^{\pm}_{S, S} = \log^2 X \prod_{\substack{p < z \\ p \mid h, p \nmid q}} \left(1-\frac{1}{p}\right) \prod_{\substack{p < z \\ p \nmid hq}} \left(1-\frac{2}{p}\right) \cdot \frac{1}{q} \sum_{\gamma \Mod{q}} \chi_0(\gamma) \chi_0(\pm \gamma +h)  \\
& \cdot  \sum_{\substack{M_1 = 2^{i_1}, M_2 = 2^{i_2} \\ 1/4 \leq M_1 \leq 4X \\ 1/4 \leq M_2 \leq 4X_{\pm}}} \sum_{\substack{m_1, m_2 \\ (m_1 m_2, P(z)) = 1}} \frac{\chi(m_1) \chi(m_2)}{m_1 m_2} \int f_{M_1, M_2}\left(\frac{m_1}{M_1}, \frac{m_2}{M_2}, \frac{y}{m_1 N_1}, \frac{\pm y+h}{m_2 N_2}\right) dy \\
&\hspace{100pt} +O_A\left(\frac{h}{\varphi(h)} X \left(e^{-Au/3000} + \frac{u^6}{z}\right) + X^{7/9+3/1000} q^2 \right)
\end{split}
\]
Taking $A = 10^6 C^2$, the error terms are acceptable by our choice of $u$ in~\eqref{eq:uchoice}. Let us denote the main term by $\widetilde{\Sigma}^{\pm}_{S, S}$. There
\[
\begin{split}
&\sum_{\substack{M_1 = 2^{i_1}, M_2 = 2^{i_2} \\ 1/4 \leq M_1 \leq X^{1/2} \\ 1/4 \leq M_2 \leq X_{\pm}^{1/2}}} \sum_{\substack{m_1, m_2 \\ (m_1 m_2, P(z)) = 1}} \frac{\chi(m_1) \chi(m_2)}{m_1 m_2}\int f_{M_1, M_2} \left(\frac{m_1}{M_1}, \frac{m_2}{M_2}, \frac{y}{m_1 N_1}, \frac{\pm y+ h}{m_2 N_2}\right) dy  \\
&= \int g\left(\frac{y}{X}\right)\sum_{\substack{m_1, m_2 \\ (m_1 m_2, P(z)) = 1}} \frac{\chi(m_1) \chi(m_2)\log\frac{y}{m_1} \log\frac{\pm y+h}{m_2}}{m_1 m_2 \cdot \log^2 X }\sum_{\substack{M_1 = 2^{i_1}, M_2 = 2^{i_2} \\ 1/4 \leq M_1 \leq X^{1/2} \\ 1/4 \leq M_2 \leq X_{\pm}^{1/2}}}  F\left(\frac{m_1}{M_1}\right) F\left(\frac{m_2}{M_2}\right) dy.
\end{split}
\]
Note that there is no dependency between $m_1$ and $m_2$,
\[
\sum_{\substack{M_1 = 2^{i_1}\\ 1/4 \leq M_1 \leq X^{1/2}}} F\left(\frac{m_1}{M_1}\right) = 
\begin{cases}
1 & \text{if $m_1 \leq X^{1/2}/4$;} \\
0 & \text{if $m_1 \geq 4X^{1/2}$,}
\end{cases}
\]
and
\[
\sum_{\substack{M_2 = 2^{i_2}\\ 1/4 \leq M_2 \leq X_{\pm}^{1/2}}} F\left(\frac{m_2}{M_2}\right) = 
\begin{cases}
1 & \text{if $m_2 \leq X_{\pm}^{1/2}/4$;} \\
0 & \text{if $m_2 \geq 4X_{\pm}^{1/2}.$}
\end{cases}
\]
Hence by partial summation and Lemmas~\ref{le:chisums} (with $v = V'/u$) and~\ref{le:Chib2}, we obtain 
\[
\begin{split}
&\widetilde{\Sigma}^{\pm}_{S, S} = \prod_{\substack{p < z \\ p \mid h, p \nmid q}} \left(1-\frac{1}{p}\right) \prod_{\substack{p < z \\ p \nmid hq}} \left(1-\frac{2}{p}\right) \prod_{p \mid (q, h)} \left(1-\frac{1}{p}\right) \prod_{\substack{p \mid q \\ p \nmid h}} \left(1-\frac{2}{p}\right) \prod_{p < z} \left(1-\frac{1}{p}\right)^{-2} \int g\left(\frac{y}{X}\right) dy \\
& \cdot \left(1+ O_A\left(\frac{u^6}{V^2 \eta^{V/(3u)}} + \frac{Vu^4}{\eta}  + \frac{u^4}{z} + e^{-A u/3000}\right) + O_{C, \varepsilon}\left(\exp(- C\log^{3/5-\eps} X)\right)\right)^2.
\end{split}
\]
Here
\begin{equation} \label{eq:productsimplify}
\begin{split}
&\prod_{\substack{p < z \\ p \mid h, p \nmid q}} \left(1-\frac{1}{p}\right) \prod_{\substack{p < z \\ p \nmid hq}} \left(1-\frac{2}{p}\right) \prod_{p \mid (q, h)} \left(1-\frac{1}{p}\right) \prod_{\substack{p \mid q \\ p \nmid h}} \left(1-\frac{2}{p}\right) \prod_{p < z} \left(1-\frac{1}{p}\right)^{-2} \\
&= 2 \prod_{2 < p < z} \frac{1-\frac{2}{p}}{\left(1-\frac{1}{p}\right)^2} \cdot \left(1+O\left(\frac{u}{z}\right)\right) \prod_{\substack{p \mid (h, q) \\ p > 2}} \frac{1-\frac{1}{p}}{1-\frac{2}{p}} \prod_{\substack{p \mid h \\ p \nmid q \\ p > 2}} \frac{1-\frac{1}{p}}{1-\frac{2}{p}} \\
&=2 \left(1+O\left(\frac{u}{z}\right)\right) \prod_{2 < p < z} \left(1-\frac{1}{(p-1)^2}\right) \prod_{\substack{p \mid h \\ p > 2}} \left(1 +  \frac{1}{p-2}\right).
\end{split}
\end{equation}

Hence
\[
\begin{split}
\widetilde{\Sigma}^{\pm}_{S, S} &= \mathfrak{S}_h \int g\left(\frac{y}{X}\right) dy \\
& + O_{A, C, \varepsilon}\left(\frac{h}{\varphi(h)} X\left(\frac{u^6}{V^{2}\eta^{V/(3u)}} + \frac{V u^4}{\eta}  +\frac{u^4}{q^{V/(2u)}} + e^{-A u/3000} + \exp(- C\log^{3/5-\eps} X) \right)\right).
\end{split}
\]
The error terms are acceptable when $A = 10^6 C^2$ by our choice of $u$ in~\eqref{eq:uchoice}.

For $\Sigma^{\pm}_{L, L}$ we use Proposition~\ref{prop:typeI2} with $\chi_1 = \chi_2 = \chi_0, \psi_1 = \psi_2 = \chi, \delta = X^{-1/1000}$ and the roles of $M_j$ and $N_j$ interchanged. Handling the error term with~\eqref{eq:ConsHenriot}, we see that, for any $A \geq 1$, 
\[
\begin{split}
\Sigma^{\pm}_{L, L} =& \log^2 X \prod_{\substack{p < z \\ p \mid h, p \nmid q}} \left(1-\frac{1}{p}\right) \prod_{\substack{p < z \\ p \nmid hq}} \left(1-\frac{2}{p}\right) \cdot \frac{1}{q} \sum_{\gamma \Mod{q}} \chi(\gamma) \chi(\pm \gamma +h)    \\
& \cdot\sum_{\substack{M_1 = 2^{i_1}, M_2 = 2^{i_2} \\ X^{1/2} < M_1 \leq 4X \\ X_{\pm}^{1/2} < M_2 \leq 4X_{\pm}}} \sum_{\substack{n_1, n_2 \\ (n_1 n_2, P(z)) = 1}} \frac{\chi(n_1) \chi(n_2)}{n_1 n_2} \int f_{M_1, M_2}\left(\frac{y}{M_1 n_1}, \frac{\pm y + h}{M_2 n_2}, \frac{n_1}{N_1}, \frac{n_2}{N_2}\right) dy \\
& \hspace{100pt} +O_A\left(\frac{h}{\varphi(h)} X \left(e^{-Au/3000} + \frac{u^6}{z}\right) + X^{7/9+3/1000} q^2\right)
\end{split}
\]
Taking $A = 10^6 C^2$, the error term is again sufficiently small by our choice of $u$ in~\eqref{eq:uchoice}. We denote the main main term by $\widetilde{\Sigma}^{\pm}_{L,L}$. Recall that $2 \mid h$ so that by Lemma~\ref{le:Chib2}
\[
\sum_{m \Mod{q}} \chi(m(\pm m+h)) = \chi(\pm 1) \sum_{m \Mod{q}} \chi_0(m)\chi_0(\pm m+h) \cdot \mathbf{1}_{\varphi(2^r) \mid h} (-1)^{\frac{h}{\varphi(2^r)}} \prod_{\substack{p \mid q' \\ p \nmid h}} \frac{-1}{p-2}.
\] 
Furthermore
\[
\begin{split}
\sum_{\substack{M_1 = 2^{i_1}, M_2 = 2^{i_2} \\ X^{1/2} < M_1 \leq 4X \\ X_{\pm}^{1/2} < M_2 \leq 4X_{\pm}}}& \sum_{\substack{n_1, n_2 \\ (n_1 n_2, P(z)) = 1}} \frac{\chi(n_1) \chi(n_2)}{n_1 n_2} \int f_{M_1, M_2}\left(\frac{y}{M_1 n_1}, \frac{\pm y + h}{M_2 n_2}, \frac{n_1}{N_1}, \frac{n_2}{N_2}\right) dy  \\
&= \int g\left(\frac{y}{X}\right)\sum_{\substack{n_1, n_2 \\ (n_1 n_2, P(z)) = 1}} \frac{\chi(n_1) \chi(n_2)\log( n_1) \log(n_2)}{n_1 n_2} \\
& \hspace{50pt} \cdot  \sum_{\substack{M_1 = 2^{i_1}, M_2 = 2^{i_2} \\ X^{1/2} < M_1 \leq 4X \\ X_{\pm}^{1/2} < M_2 \leq 4X_{\pm}}} F\left(\frac{y}{M_1 n_1}\right) F\left(\frac{\pm y+h}{M_2 n_2}\right) dy.
\end{split}
\]
Similarly to the case of $\Sigma^{\pm}_{S, S}$, we can use partial summation, Lemma \ref{le:chisums} (with $y = 1$), and (\ref{eq:productsimplify}) to obtain
\[
\begin{split}
\widetilde{\Sigma}^{\pm}_{L,L} &= \mathfrak{S}_h \chi(\pm 1) \int g\left(\frac{y}{X}\right) dy \cdot \mathbf{1}_{\varphi(2^r) \mid h} (-1)^{\frac{h}{\varphi(2^r)}} \prod_{\substack{p \mid q' \\ p \nmid h}} \frac{-1}{p-2}  \\
& +  O\left(\frac{h}{\varphi(h)} X \left(\frac{u^6}{V^{2}\eta^{V/(3u)}} + \frac{V u^4}{\eta}  +\frac{u^4}{q^{V/(2u)}} + u e^{-A u/3000} + \exp(- C\log^{3/5-\eps} X) \right)\right)
\end{split}
\]
The error term is again acceptable by (\ref{eq:uchoice}).

We handle $\Sigma^{\pm}_{S, L}$ and $\Sigma^{\pm}_{L, S}$ similarly. The error terms are the same as before whereas the main term from Proposition~\ref{prop:typeI2} is by Lemma~\ref{le:Chib2},~\eqref{eq:f(n)/n_average},~\eqref{eq:Mertens}, and~\eqref{eq:Siegeleta}
\[
\begin{split}
&\ll X\log^2 X \prod_{\substack{p < z \\ p \mid h, p \nmid q}} \left(1-\frac{1}{p}\right) \prod_{\substack{p < z \\ p \nmid hq}} \left(1-\frac{2}{p}\right) \cdot \frac{\mathbf{1}_{(h, q) = 1}}{q} \sum_{\substack{m_1 \ll X^{1/2}, m_2 \ll X_{\pm}^{1/2} \\ (m_1 m_2, P(z)) = 1}} \frac{1}{m_1 m_2} \\
&\ll_\varepsilon \frac{h}{\varphi(h)} X \frac{u^4}{q^{1-\varepsilon}} \ll \frac{h}{\varphi(h)} X \frac{1}{\eta}.
\end{split} 
\]
Collecting everything together the claims follow.

\section{Proof of Corollary~\ref{cor:shortsexzero}}
\label{se:ShortThms}
Squaring out and applying the prime number theorem, we see that
\[
\begin{split}
\int_X^{2X}& \left(\sum_{y < n \leq y+H} \Lambda(n) - H \right)^2 dy \\
&= \int_X^{2X} \left(\sum_{y < n \leq y+H} \Lambda(n) \right)^2 - 2H \sum_{y < n \leq y+H} \Lambda(n) + H^2 dy \\
&=\sum_{|h| \leq H}\sum_{\substack{n_1, n_2 \\ n_1 = n_2+h}} \Lambda(n_1) \Lambda(n_2) \int_X^{2X} \mathbf{1}_{n_1, n_1-h \in (y, y+H]} dy - H^2 X \\
&\hspace{110pt} + O_{C, \varepsilon} \left(H^3 \log X + \frac{H^2 X}{\exp(C \log^{3/5-\varepsilon} X)}\right).
\end{split}
\]
The first term on the right hand side equals
\[
\begin{split}
&\sum_{|h| \leq H} (H-|h|) \sum_{\substack{X < n_1 \leq 2X}} \Lambda(n_1) \Lambda(n_1+h) + O\left(H^3 \log^2 X\right) \\
&= H \sum_{X < n \leq 2X} \Lambda(n)^2 + \sum_{0 < |h| \leq H} (H-|h|) \sum_{X < n \leq 2X} \Lambda(n) \Lambda(n+h) + O\left(H^3 \log^2 X\right).
\end{split}
\]
The first term on the right hand side is by the prime number theorem $\ll H X \log X$. Applying Theorem~\ref{th:MTzero} to the second term, noting that $\sum_{0 < |h| \leq H} \frac{h}{\varphi(h)} \ll H$, we see that it suffices to show that 
\[
\sum_{\substack{0 < |h| \leq H \\ h \text{ even}}} (H-|h|) \mathfrak{S}_h\left(1 + \mathbf{1}_{\varphi(2^r) \mid h} (-1)^{\frac{h}{\varphi(2^r)}} \prod_{\substack{p \mid q' \\ p \nmid h}} \frac{-1}{p-2}\right)= H^2 + O\left(H \log X + \eta^{-1}H^2\right).
\]
Recall that $q'$ is necessarily square-free (see e.g.~\cite[Section 3.3]{IwKo}). Using the bound $ \mathfrak{S}_h \ll h/\varphi(h)$ and Lemma~\ref{le:gcdsum} we see that the term depending on $q$ contributes 
\[
\ll H \sum_{\substack{0 < |h| \leq H}} \frac{h}{\varphi(h)} \prod_{\substack{p \mid q' \\ p \nmid h}} \frac{1}{p-2} \ll H \prod_{\substack{p \mid q' }} \frac{1}{p-2}  \sum_{\substack{0 < |h| \leq H }} \frac{h}{\varphi(h)} (h,q') \ll \frac{H^2}{q^{1/2}},
\]
which is admissible by~\eqref{eq:Siegeleta}. Hence, it remains to show that
\begin{equation}
\label{eq:singsersumclaim}
2 \sum_{\substack{0 < |h| \leq H \\ h \text{ even}}} (H-|h|) \prod_{p > 2} \left(1-\frac{1}{(p-1)^2}\right) \prod_{\substack{p \mid h \\ p > 2}} \left(1+\frac{1}{p-2}\right) = H^2 + O\left(H \log X\right).
\end{equation}

Now
\[
\begin{split}
 \sum_{\substack{0 < |h| \leq H \\ h \text{ even}}} (H-|h|) \prod_{\substack{p \mid h \\ p > 2}} \left(1+\frac{1}{p-2}\right)& = \sum_{j = 2}^{H} \sum_{\substack{0 < |h| < j \\ h \text{ even}}} \prod_{\substack{p \mid h \\ p > 2}} \left(1+\frac{1}{p-2}\right) \\
& = \sum_{j = 2}^{H} \sum_{\substack{0 < |h| < j/2}} \prod_{\substack{p \mid h \\ p > 2}} \left(1+\frac{1}{p-2}\right).
\end{split}
\]
Define a multiplicative function $f$ such that 
\[
f(p^\nu) = \begin{cases}
0 & \text{if $p = 2$ or $\nu \geq 2$;} \\
\frac{1}{p-2} & \text{otherwise.}
\end{cases}
\]
Then
\[
\begin{split}
&\sum_{\substack{0 < |h| < j/2}} \prod_{\substack{p \mid h \\ p > 2}} \left(1+\frac{1}{p-2}\right) = \sum_{\substack{0 < |h| < j/2}} \sum_{r \mid h} f(r) = \sum_{r < j/2} f(r) \sum_{0 < |h| < j/2r} 1 \\
&= j \sum_{r < j/2} \frac{f(r)}{r} + O\left(\sum_{r < j/2} f(r)\right) = j \prod_{p > 2} \left(1+\frac{1}{p(p-2)}\right) + O(\log j) \\
&= j \prod_{p > 2} \left(1-\frac{1}{(p-1)^2}\right)^{-1} + O(\log j)
\end{split}
\]
Thus the left hand side of~\eqref{eq:singsersumclaim} equals
\[
\begin{split}
&2\sum_{j=2}^{H} j  + O\left(H \log H\right) = H^2 +O\left(H \log X\right)
\end{split}
\]
as claimed.

\section*{Acknowledgements} The first author was supported by Academy of Finland grant no. 285894. The second author was supported by Academy of Finland grant no. 333707 and  by the European Research Council (ERC) under the European Union’s Horizon 2020 research and innovation programme (grant agreement no. 851318). 

We are grateful to Terence Tao and Joni Ter\"av\"ainen for helpful discussions and to Andrew Granville for providing us material concerning the relationship between Siegel zeros and $L(1, \chi)$.

\bibliography{biblio-except}

\begin{thebibliography}{10}

\bibitem{Conrey-Iwaniec}
H.~Iwaniec B.~Conrey.
\newblock Spacing of zeros of hecke l-functions and the class number problem.
\newblock {\em Acta Arithmetica}, 103(3):259--312, 2002.

\bibitem{Opera}
J.~Friedlander and H.~Iwaniec.
\newblock {\em Opera de cribro}, volume~57 of {\em American Mathematical
  Society Colloquium Publications}.
\newblock American Mathematical Society, Providence, RI, 2010.

\bibitem{F-ISiegGold}
J.~B. Friedlander, D.~A. Goldston, H.~Iwaniec, and A.~I. Suriajaya.
\newblock Exceptional zeros and the {G}oldbach problem.
\newblock {\em J. Number Theory}, 233:78--86, 2022.

\bibitem{Illusory}
J.~B. Friedlander and H.~Iwaniec.
\newblock The illusory sieve.
\newblock {\em Int. J. Number Theory}, 1(4):459--494, 2005.

\bibitem{F-ISiegel}
J.~B. Friedlander and H.~Iwaniec.
\newblock A note on {D}irichlet {$L$}-functions.
\newblock {\em Expo. Math.}, 36(3-4):343--350, 2018.

\bibitem{G-SSiegGold}
D.~A. Goldston and A.~I. Suriajaya.
\newblock Note on the {G}oldbach conjecture and {L}andau-{S}iegel zeros.
\newblock {\em Preprint, arXiv:2104.09407v1}, 2021.

\bibitem{harman}
G.~Harman.
\newblock {\em Prime-detecting sieves}, volume~33 of {\em London Mathematical
  Society Monographs Series}.
\newblock Princeton University Press, Princeton, NJ, 2007.

\bibitem{H-BPC}
D.~R. Heath-Brown.
\newblock Gaps between primes, and the pair correlation of zeros of the zeta
  function.
\newblock {\em Acta Arith.}, 41(1):85--99, 1982.

\bibitem{hbsiegel}
D.~R. Heath-Brown.
\newblock Prime twins and {S}iegel zeros.
\newblock {\em Proc. London Math. Soc. (3)}, 47(2):193--224, 1983.

\bibitem{Henriot1}
K.~Henriot.
\newblock Nair-{T}enenbaum bounds uniform with respect to the discriminant.
\newblock {\em Math. Proc. Cambridge Philos. Soc.}, 152(3):405--424, 2012.

\bibitem{Henriot2}
K.~Henriot.
\newblock Nair-{T}enenbaum uniform with respect to the discriminant---{E}rratum
  [mr2911138].
\newblock {\em Math. Proc. Cambridge Philos. Soc.}, 157(2):375--377, 2014.

\bibitem{IwKo}
H.~Iwaniec and E.~Kowalski.
\newblock {\em Analytic number theory}, volume~53 of {\em American Mathematical
  Society Colloquium Publications}.
\newblock American Mathematical Society, Providence, RI, 2004.

\bibitem{jia}
C.~Jia.
\newblock Almost all short intervals containing prime numbers.
\newblock {\em Acta Arith.}, 76(1):21--84, 1996.

\bibitem{kloosterman}
H.~D. Kloosterman.
\newblock On the representation of numbers in the form {$ax^2+by^2+cz^2+dt^2$}.
\newblock {\em Acta Math.}, 49(3-4):407--464, 1927.

\bibitem{JoriChar}
J.~Merikoski.
\newblock Exceptional characters and prime numbers in sparse sets.
\newblock {\em Pre-print}, 2021.

\bibitem{montgomery}
H.~L. Montgomery.
\newblock The pair correlation of zeros of the zeta function.
\newblock In {\em Analytic number theory ({P}roc. {S}ympos. {P}ure {M}ath.,
  {V}ol. {XXIV}, {S}t. {L}ouis {U}niv., {S}t. {L}ouis, {M}o., 1972)}, pages
  181--193, 1973.

\bibitem{MVBook}
H.~L. Montgomery and R.~C. Vaughan.
\newblock {\em Multiplicative number theory. {I}. {C}lassical theory},
  volume~97 of {\em Cambridge Studies in Advanced Mathematics}.
\newblock Cambridge University Press, Cambridge, 2007.

\bibitem{Selberg}
A.~Selberg.
\newblock On the normal density of primes in small intervals, and the
  difference between consecutive primes.
\newblock {\em Arch. Math. Naturvid.}, 47(6):87--105, 1943.

\bibitem{TT}
T.~Tao and J.~Ter\"av\"ainen.
\newblock The {H}ardy-{L}ittlewood-{C}howla conjecture in the precence of a
  {S}iegel zero.
\newblock {\em Preprint, arXiv:2111.08912v1}, 2021.

\end{thebibliography}
\bibliographystyle{plain}

\end{document}